\documentclass[a4paper,12pt,reqno,twoside]{amsart}

\usepackage[utf8]{inputenc} 		
\usepackage[T1]{fontenc} 

\usepackage{amssymb}
\usepackage{amsmath}
\usepackage{graphicx}
\usepackage[colorlinks,citecolor=blue,urlcolor=blue]{hyperref}
\usepackage{cite}
\usepackage[hmargin=2.5cm,vmargin=2.5cm]{geometry}
\usepackage{mathrsfs} 
\usepackage{enumerate}

\newcommand{\be}{\begin{eqnarray}}
\newcommand{\ee}{\end{eqnarray}}

\newcommand{\beq}{\begin{equation}}
\newcommand{\eeq}{\end{equation}}
\newcommand{\beqn}{\begin{equation*}}
\newcommand{\eeqn}{\end{equation*}}

\newcommand{\slot}{\,\cdot\,}

\DeclareMathOperator{\Var}{\mathrm{Var}}
\DeclareMathOperator{\Cov}{\mathrm{Cov}}

\newtheorem{thm}{Theorem}[section]
\newtheorem{prop}[thm]{Proposition}
\newtheorem{cor}[thm]{Corollary}
\newtheorem{lem}[thm]{Lemma}

\newtheorem{remark}[thm]{Remark}

\newcommand\cA{{\mathcal A}}
\newcommand\cB{{\mathcal B}}

\newcommand\cE{{\mathcal E}}
\newcommand\cF{{\mathcal F}}
\newcommand\cG{{\mathcal G}}
\newcommand\cH{{\mathcal H}}

\newcommand\cN{{\mathcal N}}

\newcommand\bE{{\mathbb E}}

\newcommand\bN{{\mathbb N}}

\newcommand\bP{{\mathbb P}}
\newcommand\bQ{{\mathbb Q}}
\newcommand\bR{{\mathbb R}}

\newcommand\bZ{{\mathbb Z}}

\newcommand\rd{{\mathrm d}}

\newcommand{\id}{{\mathrm{id}}}
\newcommand{\ve}{\varepsilon}

\def\bfP{\mathbf{P}}

\begin{document}

\title[Quenched normal approximation]{Quenched normal approximation for random sequences of transformations}

\author[Olli Hella]{Olli Hella}
\address[Olli Hella]{
Department of Mathematics and Statistics, P.O.\ Box 68, Fin-00014 University of Helsinki, Finland.}
\email{olli.hella@helsinki.fi}

\author[Mikko Stenlund]{Mikko Stenlund}
\address[Mikko Stenlund]{
Department of Mathematics and Statistics, P.O.\ Box 68, Fin-00014 University of Helsinki, Finland.}
\email{mikko.stenlund@helsinki.fi}
\urladdr{http://www.helsinki.fi/~stenlund/}


 





\begin{abstract}
We study random compositions of transformations having certain uniform fiberwise properties and prove bounds which in combination with other results yield a quenched central limit theorem equipped with a convergence rate, also in the multivariate case, assuming fiberwise centering. For the most part we work with non-stationary randomness and non-invariant, non-product measures. Independently, we believe our work sheds light on the mechanisms that make quenched central limit theorems work, by dissecting the problem into three separate parts.
\end{abstract}

\maketitle


\section{The problem}\label{sec:the problem}
In the following we will study random compositions $T_{\omega_n}\circ\dots\circ T_{\omega_1}$ of maps where~$\omega = (\omega_n)_{n\ge 1}$ is a sequence drawn randomly from a probability space~$(\Omega,\cF,\bP) = (\Omega_0^{\bZ_+},\cE^{\bZ_+},\bP)$. Here $(\Omega_0,\cE)$ is a measurable space and $\bZ_+ = \{1,2,\dots\}$.
For each $\omega_0\in\Omega_0$, $T_{\omega_0}:X\to X$ is a measurable self-map on the same measurable space $(X,\cB)$. 
Consider the shift transformation
\beqn
\tau:\Omega\to\Omega: \omega = (\omega_1,\omega_2,\dots)\mapsto \tau\omega = (\omega_2,\omega_3,\dots).
\eeqn
We assume that $\tau$ is $\cF$-measurable, but does not necessarily preserve the probability measure $\bP$. Next, define the map 
\beq\label{eq:varphi}
\varphi:\bN\times\Omega\times X\to \bN\times\Omega\times X: \varphi(n,\omega,x) = T_{\omega_n}\circ\dots\circ T_{\omega_1}(x)
\eeq
with the convention $\varphi(0,\omega,x) = x$.
We assume that the map $\varphi(n,\slot,\slot)$ is measurable from $\cF\otimes\cB$ to $\cB$ for every $n\in\bN = \{0,1,\dots\}$. The maps $\varphi(n,\omega) = \varphi(n,\omega,\slot) : X\to X$ form a cocycle over the shift $\tau$, which means that the identities $\varphi(0,\omega) = \id_X$ and $\varphi(n+m,\omega) = \varphi(n,\tau^m\omega)\circ \varphi(m,\omega)$ hold. 

Consider an observable $f:X\to\bR$. Introducing notations, we write 
\beqn
f_i = f\circ T_{\omega_i}\circ\dots\circ T_{\omega_1} = f\circ\varphi(i,\omega)
\eeqn
as well as
\beqn
S_n = \sum_{i = 0}^{n-1} f_i
\quad\text{and}\quad
W_n = \frac{S_n}{\sqrt n}.
\eeqn
Given an initial probability measure $\mu$, we write $\bar f_i$ and $\bar W_n$ for the corresponding fiberwise-centered random variables:
\beqn
\bar f_i = f_i - \mu(f_i) \quad\text{and}\quad \bar W_n = W_n - \mu(W_n).
\eeqn
Note that all of these depend on $\omega$.
Next, we define
\beqn
\sigma_n^2 = \Var_\mu \bar W_n = \frac{1}{n} \sum_{i=0}^{n-1}\sum_{j=0}^{n-1} \mu(\bar f_i\bar f_j).
\eeqn
Note that $\sigma_n^2$ depends on~$\omega$.

It is said that a \emph{quenched} CLT equipped with a rate of convergence holds if there exists~$\sigma>0$ such that $d(\bar W_n, \sigma Z)$ tends to zero with the same rate \emph{for almost every~$\omega$}. Here $Z\sim\cN(0,1)$ and the limit variance $\sigma^2$ is \emph{independent} of~$\omega$. Moreover, $d$ is a distance of probability distributions which we assume to satisfy
\beqn
d(\bar W_n,\sigma Z) \le d(\bar W_n,\sigma_n Z) + d(\sigma_n Z,\sigma Z)
\eeqn
and
\beqn
d(\sigma_n Z,\sigma Z) \le C|\sigma_n - \sigma|,
\eeqn
at least when $\sigma>0$ and $\sigma_n$ is close to $\sigma$; and that $d(\bar W_n,\sigma Z)\to 0$ implies weak convergence of $\bar W_n$ to $\cN(0,\sigma^2)$.
One can find results in the recent literature that allow to bound $d(\bar W_n,\sigma_n Z)$; see~Nicol--T\"or\"ok--Vaienti~\cite{NicolTorokVaienti_2018} and Hella~\cite{Hella_2018}. In this paper we supplement those by providing conditions which allow to identify a non-random~$\sigma$ and to obtain a bound on~$|\sigma_n(\omega) - \sigma|$ which tends to zero at a certain rate for almost every~$\omega$, which is a key feature of quenched CLTs.

Our strategy is to find conditions such that $\sigma_n^2(\omega)$ converges almost surely to 
\beqn
\sigma^2 = \lim_{n\to\infty}\bE\sigma_n^2.
\eeqn
This is motivated by two observations: (1) if $\lim_{n\to\infty}\sigma_n^2 = \sigma^2$ almost surely, dominated convergence should yield the equation above, and (2) $\bE\sigma_n^2$ is the variance of $\bar W_n$ with respect to the product measure $\bP\otimes\mu$, since $\mu(\bar W_n) = 0$:
\beqn
\bE\sigma_n^2 = \bE\Var_\mu \bar W_n  = \bE\mu(\bar W_n^2) = \Var_{\bP\otimes\mu} \bar W_n.
\eeqn

\begin{remark}\label{rem:variances}One has to be careful and note that $\bar W_n$ has been centered fiberwise, with respect to~$\mu$ instead of the product measure. Therefore, $\Var_{\bP\otimes\mu} \bar W_n$ and $\Var_{\bP\otimes\mu} W_n$ differ by $\Var_\bP \mu(W_n)$:
\beqn
\bE\sigma_n^2 = \Var_{\bP\otimes\mu} \bar W_n = \bE\mu(\bar W_n^2) = \bE\Var_\mu \bar W_n = \bE\Var_\mu W_n
= \Var_{\bP\otimes\mu} W_n - \Var_\bP \mu(W_n).
\eeqn
In special cases it may happen that $\Var_\bP \mu(W_n)\to 0$, or even~$\Var_\bP \mu(W_n) = 0$ if all the maps~$T_{\omega_i}$ preserve the measure~$\mu$, whereby the distinction vanishes and the use of a non-random centering becomes feasible. We will briefly return to this point in Remark~\ref{rem:Kubo2} motivated by a result in~\cite{AbdelkaderAimino_2016}. A related observation is made in Remark~\ref{rem:Kubo1} which answers a question raised in~\cite{AiminoNicolVaienti_2015} concerning the trick of ``doubling the dimension''.
\end{remark}

To implement the strategy, we handle the terms on the right side of
\beqn
|\sigma_n^2(\omega) - \sigma^2| \le |\sigma_n^2(\omega) - \bE\sigma_n^2| + |\bE\sigma_n^2 - \sigma^2|
\eeqn
separately, obtaining convergence rates for both. Note that these are of fundamentally different type: the first one concerns almost sure deviations of $\sigma^2_n$ about the mean, while the second one concerns convergence of said mean together with identification of the limit.

\begin{remark}
That the required bounds can be obtained illuminates the following pathway to a quenched central limit theorem:
\vspace{-1mm}
\begin{enumerate}
\item $d(\bar W_n,\sigma_n Z) \to 0$ almost surely,
\item $\sigma^2_n - \bE\sigma_n^2 \to 0$ almost surely,
\item $\bE\sigma_n^2 \to \sigma^2$ for some $\sigma^2>0$,
\end{enumerate}
where the last step involves identification of~$\sigma^2$.
\end{remark}

\begin{remark} Let us emphasize that in general we do not assume $\bP$ to be stationary or of product form; $\mu$ to be invariant for any of the maps $T_{\omega_i}$; or $\bP\otimes\mu$ (or any other measure of similar product form) to be invariant for the random dynamical system associated to the cocycle~$\varphi$.
\end{remark}

\medskip

Quenched limit theorems for random dynamical systems are abundant in the literature, going back at least to Kifer~\cite {Kifer_1998}. Nevertheless they remain a lively topic of research to date: Recent central limit theorems and invariance principles in such a setting include Ayyer--Liverani--Stenlund~\cite{AyyerLiveraniStenlund_2009}, Nandori--Szasz--Varju~\cite{NandoriSzaszVarju_2012}, Aimino--Nicol--Vaienti~\cite{AiminoNicolVaienti_2015}, Abdelkader--Aimino~\cite{AbdelkaderAimino_2016}, Nicol--T\"or\"ok--Vaienti~\cite{NicolTorokVaienti_2018}, Dragi{\v{c}}evi{\'c} et al.~\cite{DragicevicFroylandGonzalez-TokmanVaienti_2018, Dragicevic_etal_2018}, and Chen--Yang--Zhang~\cite{Chen_2018}. Moreover, Bahsoun et al.~\cite{BahsounBoseDuan_2014, BahsounBose_2016, BahsounBoseRuziboev_2017} establish important optimal quenched correlation bounds with applications to limit results, and Freitas--Freitas--Vaienti~\cite{FreitasFreitasVaienti_2017} establish interesting extreme value laws which have attracted plenty of attention during the past years.

\medskip
\noindent{\bf Structure of the paper.} The main result of our paper is Theorem~\ref{thm:main} in Section~\ref{sec:main}. It is an immediate corollary of Theorem~\ref{thm: sigma_n^2} of Section~\ref{sec:conv of sigma_n^2}, which concerns $|\sigma_n^2(\omega) - \bE\sigma_n^2|$, and of Theorem~\ref{thm: sigma_n^2 to sigma^2} of Section~\ref{sec:conv of sigma_n^2 to sigma^2}, which concerns $|\bE\sigma_n^2 - \sigma^2|$. In Section~\ref{sec:main} we also explain how the results of this paper extend to the vector-valued case $f:X\to\bR^d$. 

At the end of the paper the reader will find several appendices, which are integral parts of the paper: In Appendix~\ref{sec:RDS} we interpret the limit variance~$\sigma^2$ in the language of random dynamical systems and skew products. In Appendix~\ref{sec:pos_var} we present conditions for $\sigma^2>0$. In Appendix~\ref{sec:centerings}, we discuss how the fiberwise centering in the definition of $\bar W_n$ affects the limit variance.  For completeness, in Appendix~\ref{sec:SA5'} we elaborate on the structure of an invariant measure intimately related to the problem.



\section{The term $|\sigma_n^2(\omega) - \bE\sigma_n^2|$}\label{sec:conv of sigma_n^2}
In this section identify conditions which guarantee that, almost surely, $|\sigma_n^2(\omega) - \bE\sigma_n^2|$ tends to zero at a specific rate.

\medskip
\noindent{\bf Standing Assumption (SA1).} Throughout this paper we will assume that $f$ is a bounded measurable function and $\mu$ is a probability measure. We also assume that a uniform decay of correlations holds in that 
\beqn
|\mu(\bar f_i\bar f_j)| \le \eta(|i-j|)
\eeqn
almost surely, where $\eta:\bN\to[0,\infty)$ is such that 
\beq\label{eq:weak_eta}
\sum_{i=0}^\infty \eta(i) < \infty
\quad\text{and}\quad
\text{$\eta$ is non-increasing}.
\eeq
\hfill$\blacksquare$
\medskip

Note already that
\beqn
\lim_{n\to\infty}\frac1n\sum_{i=1}^n i\eta(i) = 0
\eeqn
because $\lim_{n\to\infty}i\eta(i) = 0$. For the most part, we shall require additional conditions on~$\eta$.

For future convenience, let us introduce the random variables
\beqn
v_i = v_i(\omega) = \sum_{j=i}^\infty (2-\delta_{ij})\mu(\bar f_i\bar f_j)
\eeqn
and their centered counterparts
\beqn
\tilde v_i = v_i - \bE v_i.
\eeqn
Note that these are uniformly bounded. We also denote
\beqn
\tilde\sigma_n^2 = \sigma_n^2 - \bE\sigma_n^2.
\eeqn
Thus, our objective is to show $\tilde\sigma_n^2 \to 0$ at some rate.

The following lemma is readily obtained by a well-known computation:
\begin{lem}\label{lem:var_mean}
Assuming~\eqref{eq:weak_eta}, there exists a constant $C>0$ such that
\beqn
\left|\sigma_n^2 - \frac{1}{n} \sum_{i=0}^{n-1} v_i\right| \le C\!\left(\frac1n\sum_{i=1}^n i\eta(i) + \sum_{i=n+1}^\infty\eta(i)\right) = o(1)
\eeqn
for all $\omega$.
\end{lem}

\begin{proof}
First, we compute
\beqn
\begin{split}
\sigma_n^2 & = \frac{1}{n} \sum_{i=0}^{n-1}\sum_{j=0}^{n-1} \mu(\bar f_i\bar f_j) = \frac{1}{n} \! \left[\sum_{i=0}^{n-1} \mu(\bar f_i^2) + 2 \sum_{0\le i<j<n}\mu(\bar f_i\bar f_j)\right]
\\
& = \frac{1}{n} \sum_{i=0}^{n-1}\! \left[\mu(\bar f_i^2) + 2 \sum_{j = i+1}^{n-1} \mu(\bar f_i\bar f_j)\right]
\\
& = \frac{1}{n} \sum_{i=0}^{n-1}\! \left[\mu(\bar f_i^2) + 2 \sum_{j = i+1}^\infty \mu(\bar f_i\bar f_j)\right]  + O\!\left( \frac{1}{n} \sum_{i=0}^{n-1}\sum_{j = n}^\infty\eta(j-i) \right).
\end{split}
\eeqn
Here
\beqn
 \frac{1}{n} \sum_{i=0}^{n-1}\sum_{j = n}^\infty\eta(j-i) =  \frac{1}{n} \sum_{i=0}^{n-1}\sum_{j = n}^{n+i}\eta(j-i) + \frac{1}{n} \sum_{i=0}^{n-1}\sum_{j = n+i+1}^\infty\eta(j-i) = \frac1n\sum_{i=1}^n i\eta(i) + \sum_{i=n+1}^\infty\eta(i).
\eeqn
The last sums tend to zero by assumption.
\end{proof}

Suppose that $\eta(0)=A$ and $\eta(n) = An^{-\psi}$, $n\ge 1$, for some constants $A\ge 0,\psi>0$. We then use shorthand notation $\eta(n) = An^{-\psi}$, i.e., we interpret $0^{-\psi}=1$.
\begin{cor}\label{cor:var_mean_poly}
Suppose $\eta(n) = An^{-\psi}$, where $\psi>1$. Then
\beqn
\left|\sigma_n^2 - \frac{1}{n} \sum_{i=0}^{n-1} v_i\right| \le C 
\begin{cases}
n^{-1}, & \psi > 2,\\
n^{-1}\log n, & \psi = 2,\\
n^{1-\psi}, & 1<\psi<2.
\end{cases}
\eeqn
\end{cor}
We skip the elementary proof based on Lemma~\ref{lem:var_mean}.

\begin{remark}\label{rem:alsotilde}
Of course, the upper bounds in the preceding results apply equally well to
\beqn
\tilde\sigma_n^2 - \frac1n\sum_{i=0}^{n-1} \tilde v_i.
\eeqn
\end{remark}

The following result, which has been used in dynamical systems papers including Melbourne--Nicol~\cite{MelbourneNicol_2009}, will be used to obtain an almost sure convergence rate of $\frac1n\sum_{i=0}^{n-1} \tilde v_i$ to zero:

\begin{thm}\label{thm:GK}[G\'al--Koksma~\cite{GalKoksma_1950}; see also Philipp--Stout~\cite{PhilippStout_1975}] Let $(X_n)$ be a sequence of centered, square-integrable, random variables. Suppose there exist $C>0$ and $q>0$ such that
\beqn
\bE\!\left[\left(\sum_{k=m}^{m+n-1} X_k\right)^2\right] \le C[(n+m)^q - m^q]
\eeqn
for all $m\ge 0$ and $n\ge 1$. Let $\delta>0$ be arbitrary. Then, almost surely,
\beqn
\frac1n\sum_{k=1}^n X_k = O(n^{\frac q2-1} \log^{\frac32+\delta}n).
\eeqn
\end{thm}

\begin{remark}\label{rmrk:q}
In this paper the theorem is applied in the range $1\le q<2$. In particular, $n^q + m^q \le (n+m)^q$ then holds, so it suffices to establish an upper bound of the form~$Cn^q$. 
\end{remark}

Our application of Theorem~\ref{thm:GK} will be based on the following standard lemma:
\begin{lem}\label{lem:GK_correlation}
Suppose $|\bE[X_iX_k]| \le r(|k-i|)$ where $r(k) = O(k^{-\beta})$. There exists a constant~$C>0$ such that
\beqn
\bE\!\left[\left(\sum_{k=m}^{m+n-1} X_k\right)^2\right] \le
C\begin{cases}
n, & \beta>1,
\\
n\log n, & \beta = 1,
\\
n^{2-\beta}, & 0<\beta<1.
\end{cases}
\eeqn
\end{lem}

\begin{proof}
Note that
\beqn
\begin{split}
\bE\!\left[\left(\sum_{k=m}^{m+n-1} X_k\right)^2\right]
& \le\sum_{k=m}^{m+n-1}\sum_{l=m}^{m+n-1}r(|k-l|)= n r(0) + \sum_{k=1}^{n-1} 2(n-k)r(k)
\\
& \le  n r(0) + 2n\sum_{k=1}^{n-1} r(k) \le Cn\sum_{k=1}^{n-1} k^{-\beta}.
\end{split}
\eeqn
Bounding the last sum in each case yields the result.
\end{proof}

\subsection{Dependent random selection process}\label{sec:selection_process}
It is most interesting to study the case where the sequence~$\omega = (\omega_i)_{i\ge 1}$ is generated by a non-trivial stochastic process such that the measure $\bP$ is not the product of its one-dimensional marginals. Essentially without loss of generality, we pass directly to the so-called canonical version of the process, which corresponds to the point of view that the sequence~$\omega$ \emph{is} the seed of the random process. In the following we briefly review some standard details.

Let $\pi_i:\Omega\to \Omega_0$ be the projection $\pi_i(\omega) = \omega_i$. The product sigma-algebra $\cF$ is the smallest sigma-algebra with respect to which all the latter projections are measurable. For any $I = (i_1,\dots,i_p)\subset {\bZ_+}$, $p\in{\bZ_+}\cup\{\infty\}$, we may define the sub-sigma-algebra $\cF_I = \sigma(\pi_i:i\in I)$ of~$\cF$. (In particular, $\cF = \cF_{\bZ_+}$.) We also recall that a function $u:\Omega\to\bR$ is $\cF_I$-measurable if and only if there exists an $\cE^p$-measurable function $\tilde u:\Omega_0^p\to\bR$ such that $u = \tilde u\circ (\pi_{i_1},\dots,\pi_{i_p})$, i.e., $u(\omega) = \tilde u(\omega_{i_1},\dots,\omega_{i_p})$. With slight abuse of language, we will say below that the sigma-algebra $\cF_I$ is generated by the random variables $\omega_i$, $i\in I$, instead of the projections $\pi_i$.
In particular, we denote 
\beqn
\cF_i^j = \sigma(\omega_n:i\le n\le j)\subset\cF
\eeqn
for $1\le i\le j\le \infty$. 

Denote
\beqn
\alpha(\cF_1^i,\cF_j^\infty) = \sup_{A\in \cF_1^i,\, B\in \cF_j^\infty}|\bP(AB) - \bP(A)\,\bP(B)|.
\eeqn
In the following $(\alpha(n))_{n\ge 1}$ will denote a sequence such that
\beqn
\sup_{i\ge 1}\alpha(\cF_1^i,\cF_{i+n}^\infty) \le \alpha(n)
\eeqn
for each $n\ge 1$. 

\medskip
\noindent{\bf Standing Assumption (SA2).}
Throughout the rest of the paper we assume that the random selection process is {\bf strong mixing}: $\alpha(n)$ can be chosen so that\footnote{It would be standard to denote $\sup_{i\ge 1}\alpha(\cF_1^i,\cF_{i+n}^\infty)$ by $\alpha(n)$. We prefer to let $\alpha(n)$ stand for an upper bound so the non-increasing assumption makes sense. This is a choice of technical convenience.}
\beqn
\lim_{n\to\infty}\alpha(n)=0 \quad\text{and}\quad\text{$\alpha$ is non-increasing.}
\eeqn
\hfill$\blacksquare$
\medskip

Suppose that $u = u(\omega_1,\dots,\omega_i)$ and $v = v(\omega_{i+n},\omega_{i+n+1},\dots)$ are~$L^\infty$ functions. Then
\beq\label{eq:strongmixing}
|\bE[uv] - \bE u\bE v| \le 4\|u\|_\infty\|v\|_\infty\alpha(n)
\eeq
as is well known.
Ultimately, we will impose a rate of decay on $\alpha(n)$.


We denote by $T_*$ the pushforward of a map $T$, acting on a probability measure $m$, i.e., $(T_*m)(A) = m(T^{-1}A)$ for measurable sets $A$. We write
\beqn
\mu_k = (T_{\omega_{k}}\circ \dots \circ T_{\omega_{1}})_{*}\mu
\eeqn
and
\beqn
\mu_{k,r+1}=(T_{\omega_{k}}\circ \dots \circ T_{\omega_{r+1}})_{*}\mu
\eeqn
for $k\ge r$. We also write
\beqn
f_{l,k+1} = f\circ T_{\omega_l}\circ\dots\circ T_{\omega_{k+1}} = f\circ \varphi(l-k,\tau^k\omega)
\eeqn
for $l\ge k$. Note that all of these objects depend on~$\omega$ through the maps~$T_{\omega_i}$. We use the conventions $\mu_0 = \mu$, $\mu_{r,r+1} = \mu$ and $f_{k,k+1} = f$ here.

\medskip
\noindent{\bf Standing Assumption (SA3).} Throughout the rest of the paper we assume the following uniform {\bf memory-loss condition}: there exists a constant $C\ge 0$ such that
\beq\label{eq:memloss}
|\mu_{k}(g)-\mu_{k,r+1}(g)| \le C\eta(k-r)
\eeq
for all
\beqn
g\in \cG_k = \cG_k(\omega) = \{f_{l,k+1}:l\ge k\} \cup\{ff_{l,k+1}:l\ge k \}
\eeqn
whenever $k\ge r$. The bound holds uniformly for (almost) all~$\omega$.
\hfill$\blacksquare$
\medskip

In the cocycle notation,~\eqref{eq:memloss} reads
\beq\label{eq:memloss_cocycle}
|\mu(g\circ\varphi(k,\omega)) - \mu(g\circ\varphi(k-r,\tau^r\omega))| \le C\eta(k-r).
\eeq

Note that, setting
\beqn
\tilde c_{ij} = (2-\delta_{ij})[\mu(\bar f_i\bar f_j) - \bE\mu(\bar f_i\bar f_j)],
\eeqn
we have
\beqn
\tilde v_i = \sum_{j=i}^\infty \tilde c_{ij}
\quad\text{and}\quad
 \bE[\tilde v_i\tilde v_k] = \bE\!\left[\left(\sum_{j=i}^\infty \tilde c_{ij}\right)\left(\sum_{l=k}^\infty \tilde c_{kl}\right)\right].
\eeqn

\begin{lem}\label{lem:tildec_decorr_mixing}
There exists a constant $C\ge 0$ such that
\beqn
|\bE[\tilde c_{ij}\tilde c_{kl}]| 
\le
\begin{cases}
C\eta(j-i)\eta(l-k), & \text{if $i\le j$ and $k\le l$}\\
C\eta(j-i)\min_{r:j\le r\le k}\{\eta(k-r) + \alpha(r-j)\eta(l-k)\}, & \text{if $i\le j \le k\le l$}.
\end{cases}
\eeqn
In particular, for $i\le j \le k\le l$,
\beqn
|\bE[\tilde c_{ij}\tilde c_{kl}]| \le C\eta(j-i) \min\!\left\{\eta(l-k),\min_{r:j\le r\le k}\{\eta(k-r) + \alpha(r-j)\eta(l-k)\}\right\}.
\eeqn
\end{lem}

\begin{proof}
The first bound holds, because
\beqn
|\tilde c_{ij}\tilde c_{kl}| \le 2\eta(|j-i|)\quad \text{ and } \quad|\tilde c_{ij}\tilde c_{kl}| \le 2\eta(|l-k|).
\eeqn
Suppose $ i\le j \le r \le k\le l$ holds. By (SA3), the choice $g=f f_{l,k+1}$ yields
\beqn
|\mu(f_{k}f_{l}) - \mu_{k,r+1}(ff_{l,k+1})| \le C\eta(k-r),
\eeqn
while the choices $g = f$ and $g = f_{l,k+1}$ together yield
\beqn
|\mu(f_{k})\mu(f_{l}) - \mu_{k,r+1}(f)\mu_{k,r+1}(f_{l,k+1})| \le C\eta(k-r).
\eeqn
Hence
\beq\label{eq:tildec_decorr_mixing_proof1}
| \mu(\bar{f_{k}}\bar{f_{l}}) - \{\mu_{k,r+1}(ff_{l,k+1}) - \mu_{k,r+1}(f)\mu_{k,r+1}(f_{l,k+1})\} | \le C\eta(k-r).
\eeq
Note that here the expression in the curly braces only depends on the random variables $\omega_{r+1},\dots,\omega_l$ while $\mu(\bar f_i \bar f_j)$ only depends on $\omega_1,\dots,\omega_j$. More precisely, denoting $u = \mu(\bar f_i \bar f_j)$ and $v = \mu_{k,r+1}(ff_{l,k+1}) - \mu_{k,r+1}(f)\mu_{k,r+1}(f_{l,k+1})$, we have $u\in L^\infty(\cF_1^j)$ and $v\in L^\infty(\cF_{r+1}^l)\subset L^\infty(\cF_r^\infty)$.
Therefore,
\beqn
|\bE[\mu(\bar{f_{i}}\bar{f_{j}})\mu(\bar{f_{k}}\bar{f_{l}})] - \bE[uv]| \le C\bE[|u|]\eta(k-r) \le C\eta(j-i)\eta(k-r)
\eeqn
by \eqref{eq:tildec_decorr_mixing_proof1}. On the other hand, the strong-mixing bound~\eqref{eq:strongmixing} implies
\beqn
|\bE[uv] - \bE u\,\bE v| \le \alpha(r-j)\|u\|_\infty\|v\|_\infty \le C\alpha(r-j)\eta(j-i)\|v\|_\infty.
\eeqn
Moreover,
\beqn
|\bE[\mu(\bar{f_{i}}\bar{f_{j}})]\bE[\mu(\bar{f_{k}}\bar{f_{l}})] - \bE u\,\bE v| \le |\bE u| |\bE[\mu(\bar{f_{k}}\bar{f_{l}}) - v]| \le C\eta(j-i)\eta(k-r).
\eeqn
Collecting the bounds leads to the estimate
\beqn
\begin{split}
|\bE[\tilde c_{ij}\tilde c_{kl}]| 
& \le 4 |\bE[\mu(\bar{f_{i}}\bar{f_{j}})\mu(\bar{f_{k}}\bar{f_{l}})]-\bE[\mu(\bar{f_{i}}\bar{f_{j}})]\bE[\mu(\bar{f_{k}}\bar{f_{l}})]|
\\
& \le C\eta(j-i)\{\eta(k-r) + \alpha(r-j)\|v\|_\infty\}.
\end{split}
\eeqn
Note that~\eqref{eq:tildec_decorr_mixing_proof1} immediately yields the estimate
\beqn
\|v\|_\infty \le C\eta(l-k) + C\eta(k-r)
\eeqn
which by the boundedness of $\alpha$ results in
\beqn
\begin{split}
|\bE[\tilde c_{ij}\tilde c_{kl}]| & \le C\eta(j-i)\{\eta(k-r) + \alpha(r-j)[\eta(l-k) + \eta(k-r)]\}
\\
& \le C\eta(j-i)\{\eta(k-r) + \alpha(r-j)\eta(l-k)\}.
\end{split}
\eeqn
Taking the minimum with respect to $r$ proves the lemma.
\end{proof}

The upper bound $|\bE[\tilde c_{ij}\tilde c_{kl}]| \le C\eta(j-i)\eta(l-k)$ of Lemma~\ref{lem:tildec_decorr_mixing} yields the following intermediate result:
\begin{lem}\label{lem:strong_mixing_bound}
For $i\le k$,
\beqn
|\bE[\tilde v_i\tilde v_k]|\le
C\!\left(\sum_{j=i}^{k-1}\sum_{l=k}^{2k-j-1} |\bE[\tilde c_{ij}\tilde c_{kl}]|
 + \sum_{n=m}^\infty \eta(n)
+ \sum_{n=0}^{m-1}\sum_{p=m-n}^{\infty} \eta(n)\eta(p)
\right),
\eeqn
where we have denoted $m=k-i$.
\end{lem}
\begin{proof}
We can estimate
\beqn
\begin{split}
|\bE[\tilde v_i\tilde v_k]|
& \le \sum_{j=i}^\infty \sum_{l=k}^\infty |\bE[\tilde c_{ij}\tilde c_{kl}]|
\\
& = \sum_{j=i}^{k-1}\sum_{l=k}^{2k-j-1} |\bE[\tilde c_{ij}\tilde c_{kl}]|
+ \sum_{j=i}^{k-1}\sum_{l=2k-j}^\infty |\bE[\tilde c_{ij}\tilde c_{kl}]|
+ \sum_{j=k}^\infty\sum_{l=k}^\infty |\bE[\tilde c_{ij}\tilde c_{kl}]|
\\
& \le \sum_{j=i}^{k-1}\sum_{l=k}^{2k-j-1} |\bE[\tilde c_{ij}\tilde c_{kl}]| 
+ \sum_{j=i}^{k-1}\sum_{l=2k-j}^\infty C\eta(j-i)\eta(l-k)
+ \sum_{j=k}^\infty\sum_{l=k}^\infty C\eta(j-i)\eta(l-k)
\\
& \le \sum_{j=i}^{k-1}\sum_{l=k}^{2k-j-1} |\bE[\tilde c_{ij}\tilde c_{kl}]| 
+ C\sum_{n=0}^{k-i-1}\sum_{p=k-i-n}^{\infty} \eta(n)\eta(p) + C\sum_{n=k-i}^\infty \eta(n).
\end{split}
\eeqn
In the third line we used the upper bound $|\bE[\tilde c_{ij}\tilde c_{kl}]| \le C\eta(j-i)\eta(l-k)$ of Lemma~\ref{lem:tildec_decorr_mixing}.
\end{proof}

Next we investigate the remaining term
$
\sum_{j=i}^{k-1}\sum_{l=k}^{2k-j-1} |\bE[\tilde c_{ij}\tilde c_{kl}]|
$
appearing in Lemma~\ref{lem:strong_mixing_bound}.
Since $i\le j\le k\le l$, we have
\beqn
\min_{r:j\le r\le k}\{\eta(k-r) + \alpha(r-j)\eta(l-k)\} \le \eta(k-j) + \alpha(0)\eta(l-k)
\eeqn
by choosing $r = j$. Suppose furthermore that $k-j \ge l-k$ and recall $\eta$ is non-increasing. Then the right side of the above display is bounded above by $C\eta(l-k)$. In other words, if $i\le j\le k\le l \le 2k-j$, then
 $C\eta(j-i)\min_{r:j\le r\le k}\{\eta(k-r) + \alpha(r-j)\eta(l-k)\}$ is the tightest bound on $|\bE[\tilde c_{ij}\tilde c_{kl}]|$ that Lemma~\ref{lem:tildec_decorr_mixing} can provide. 
This observation motivates the following lemma.

\begin{lem}\label{lem:Sbounds}
Define
\beqn
S(i,k) = \sum_{j=i}^{k-1} \eta(j-i)\sum_{l=k}^{2k-j-1} \min_{r:j\le r\le k}\{\eta(k-r) + \alpha(r-j)\eta(l-k)\}.
\eeqn
{\bf (i)} There exists a constant $C\ge 0$ such that
\beqn
\sum_{j=i}^{k-1}\sum_{l=k}^{2k-j-1} |\bE[\tilde c_{ij}\tilde c_{kl}]|\le CS(i,k)
\eeqn
whenever $i\le k$. 

\noindent {\bf (ii)} There exist constants $C_1\ge 0$ and $C_2\ge 0$ such that
\beqn
C_{1}\{m\eta(m)+ \alpha(m)\}\le S(i,k) \le  C_{2}\! \left\lbrace m\eta\!\left(\left\lfloor\frac{m}{4}\right\rfloor\right)+  \alpha\!\left(\left\lfloor\frac{m}{4}\right\rfloor\right)\right\rbrace \qquad (m = k-i)
\eeqn
whenever $i<k$. (Note also that $S(i,i) = 0$.)
\end{lem}

\begin{proof}
Part (i) is an immediate corollary of Lemma~\ref{lem:tildec_decorr_mixing}. As for part (ii), let us first prove the lower bound.
Since all the terms in $S(i,k)$ are nonnegative and $\alpha$ is non-increasing, we have for $i<k$ that
\beqn
S(i,k)\ge \sum_{j=i}^{i} \eta(j-i)\sum_{l=k}^{k} \min_{r:j\le r\le k}\alpha(r-j)\eta(l-k)\ge \eta(0)^{2}\alpha(k-i)
\eeqn
and
\beqn
S(i,k)\ge \sum_{j=i}^{i} \eta(j-i)\sum_{l=k}^{2k-j-1} \min_{r:j\le r\le k}\eta(k-r) \ge \eta(0) (k-i)\eta(k-i). 
\eeqn
Setting $C_1 = \frac12\eta(0)^{2} +\frac12 \eta(0)$ gives an overall bound
\beqn
S(i,k)\ge C_1\{\alpha(m) + m\eta(m)\}
\eeqn
for all $i<k$.

It remains to prove the upper bound in part (ii).
 We choose $r=\lfloor(k+j)/2\rfloor$. Since $\eta$ is summable, we have
 \begin{align*}
S(i,k) = \sum_{j=i}^{k-1} \eta(j-i)\sum_{l=k}^{2k-j-1} \left\lbrace\eta\!\left(k-\left\lfloor\frac{k+j}{2}\right\rfloor\right) + \alpha\!\left(\left\lfloor\frac{k+j}{2}\right\rfloor-j\right)\eta(l-k)\right\rbrace
 \\
 \le C \sum_{j=i}^{k-1} \eta(j-i) \left\lbrace(k-j)\eta\!\left(\left\lfloor\frac{k-j}{2}\right\rfloor\right) + \alpha\!\left(\left\lfloor\frac{k-j}{2}\right\rfloor\right)\right\rbrace
 \\
 = C \sum_{j=0}^{m-1} \eta(j) \left\lbrace(m-j)\eta\!\left(\left\lfloor\frac{m-j}{2}\right\rfloor\right) + \alpha\!\left(\left\lfloor\frac{m-j}{2}\right\rfloor\right)\right\rbrace.
 \end{align*}
 Next we split the last sum above into two parts, keeping in mind that $\alpha$ and $\eta$ are non-increasing and $\eta$ is also summable:
\begin{align*}
&C \sum_{j=0}^{m-1} \eta(j) \left\lbrace(m-j)\eta\!\left(\left\lfloor\frac{m-j}{2}\right\rfloor\right) + \alpha\!\left(\left\lfloor\frac{m-j}{2}\right\rfloor\right)\right\rbrace
\\
&\le C \sum_{j=0}^{\lfloor m/2 \rfloor} \eta(j) \left\lbrace(m-j)\eta\!\left(\left\lfloor\frac{m-j}{2}\right\rfloor\right) + \alpha\!\left(\left\lfloor\frac{m-j}{2}\right\rfloor\right)\right\rbrace 
\\
&+ C \sum_{j=\lfloor m/2 \rfloor + 1}^{m-1} \eta(j) \left\lbrace(m-j)\eta\!\left(\left\lfloor \frac{m-j}{2} \right\rfloor\right) + \alpha\!\left(\left\lfloor\frac{m-j}{2}\right\rfloor\right)\right\rbrace 
\\
&\le  C \sum_{j=0}^{\lfloor m/2\rfloor} \eta(j) \left\lbrace m\eta\!\left(\left\lfloor\frac{m}{4}\right\rfloor\right) + \alpha\!\left(\left\lfloor\frac{m}{4}\right\rfloor\right)\right\rbrace 
\\
&+ 
 C \eta\!\left(\left\lfloor\frac{m}{2}\right\rfloor\right) \sum_{j=m/2}^{m}\left\lbrace m\eta\!\left(\left\lfloor\frac{m-j}{2}\right\rfloor\right) + \alpha\!\left(\left\lfloor\frac{m-j}{2}\right\rfloor\right) \right\rbrace
\\
&\le   C \left\lbrace m\eta\!\left(\left\lfloor\frac{m}{4}\right\rfloor\right)+  \alpha\!\left(\left\lfloor\frac{m}{4}\right\rfloor\right)\right\rbrace +  Cm \eta\!\left(\left\lfloor\frac{m}{2}\right\rfloor\right)\le C_2 \left\lbrace m\eta\!\left(\left\lfloor\frac{m}{4}\right\rfloor\right)+  \alpha\!\left(\left\lfloor\frac{m}{4}\right\rfloor\right)\right\rbrace.
\end{align*}
This completes the proof.
\end{proof} 

The next two lemmas concern the case when $\eta$ and $\alpha$ are polynomial.

\begin{lem}\label{lem:Sbounds_poly}
Let $\eta(n)=Cn^{-\psi}$, $\psi>1$ and $\alpha(n)=Cn^{-\gamma}$, $\gamma>0$. Then
\beqn
 C_{1}m^{-\min\{\psi-1,\gamma\}} \le S(i,k)\le C_{2}m^{-\min\{\psi-1,\gamma\}} \qquad (m = k-i)
\eeqn
whenever $i<k$.

\end{lem}
\begin{proof}
The lower bound follows immediately from Lemma~\ref{lem:Sbounds}(ii). Let first $m\ge 8$. Then~$\lfloor m/4\rfloor \ge m/8$. Thus Lemma~\ref{lem:Sbounds}(ii) yields 
\beqn
 S(i,k) \le  C\! \left\lbrace m\eta\!\left(\frac{m}{8}\right)+  \alpha\!\left(\frac{m}{8}\right)\right\rbrace \le  Cm^{\max\{1-\psi,-\gamma\}},
\eeqn
when $m\ge 8$. Since $S(i,k)\le C(k-i)^2 = Cm^2$ by counting terms, we can choose a large enough $C_{2}$ such that the claimed upper bound holds also for $1\le m<8$. 
\end{proof}

\begin{lem}\label{lem:rbounds}
Let $\eta(n)=Cn^{-\psi}$, $\psi>1$ and $\alpha(n)=Cn^{-\gamma}$, $\gamma>0$.  Then
\beqn
|\bE[\tilde v_i\tilde v_k]| \le Cm^{-\min\{\psi-1,\gamma\}} \qquad (m = k-i)
\eeqn
whenever $i<k$.

\end{lem}
\begin{proof}
Firstly,
\beq
\sum_{n=m}^\infty \eta(n) \le Cm^{1-\psi}.\label{eq:rbound1}
\eeq
Secondly,
\beqn
\begin{split}
\sum_{n=0}^{m-1}\sum_{p=m-n}^\infty \eta(n)\eta(p) 
& = \sum_{p=m}^\infty \eta(p) + \sum_{n=1}^{m-1}\sum_{p=m-n}^\infty \eta(n)\eta(p)
\\
& = C\sum_{p=m}^\infty p^{-\psi} + C\sum_{n=1}^{m-1}\sum_{p=m-n}^\infty n^{-\psi}p^{-\psi}
\\
& \le Cm^{1-\psi} + C\sum_{n=1}^{m-1} n^{-\psi}(m-n)^{1-\psi}
\\
& = Cm^{1-\psi} + C\sum_{n=1}^{m-1} n^{1-\psi}(m-n)^{-\psi}.
\end{split}
\eeqn
Regarding the last sum appearing above, observe that
\beqn
\begin{split}
\sum_{n=1}^{m/2} n^{1-\psi} (m-n)^{-\psi} 
\le \sum_{n=1}^{m/2} 1^{1-\psi} (m/2)^{-\psi} \le Cm^{1-\psi},
\end{split}
\eeqn
while
\beqn
\begin{split}
\sum_{n=m/2}^{m-1} n^{1-\psi} (m-n)^{-\psi} 
\le \sum_{n=m/2}^{m-1} (m/2)^{1-\psi} (m-n)^{-\psi}
\le (m/2)^{1-\psi} \sum_{n=1}^{m/2} n^{-\psi} \le Cm^{1-\psi}.
\end{split}
\eeqn
In other words, also
\beq\label{eq:rbound2}
\sum_{n=0}^{m-1}\sum_{p=m-n}^\infty \eta(n)\eta(p)  \le Cm^{1-\psi}.
\eeq
Now, by Lemmas~\ref{lem:Sbounds}(i) and~\ref{lem:Sbounds_poly}
we have
\beqn
\sum_{j=i}^{k-1}\sum_{l=k}^{2k-j-1} |\bE[\tilde c_{ij}\tilde c_{kl}]|\le S(i,k) \le Cm^{\max\{1-\psi,-\gamma\}}.
\eeqn
Thus, Lemma \ref{lem:strong_mixing_bound} and bounds \eqref{eq:rbound1} and \eqref{eq:rbound2} yield
\beqn
|\bE[\tilde v_i\tilde v_k]|\le Cm^{1-\psi}+ Cm^{\max\{1-\psi,-\gamma\}}\le Cm^{\max\{1-\psi,-\gamma\}}. 
\eeqn
The proof is complete.
\end{proof}

\begin{lem}\label{lem:tilv_sum2}
Suppose $|\bE[\tilde{v}_i\tilde{v}_k]| \le C(k-i)^{-\beta}$ for all $i<k$. Let $\delta>0$ be arbitrary. Then
\beqn
\frac1n\sum_{k=1}^n \tilde{v}_k = 
\begin{cases}
O(n^{-\frac 12} \log^{\frac32+\delta}n), & \beta > 1,
\\
O(n^{-\frac12+\delta}), & \beta = 1,
\\
O(n^{-\frac{\beta}{2}} \log^{\frac32+\delta}n), & 0<\beta<1
\end{cases}
\eeqn
almost surely.
\end{lem}
\begin{proof}
Applying Lemma \ref{lem:GK_correlation} we get
\beqn
\bE\!\left[\left(\sum_{k=m}^{m+n-1} \tilde{v}_k\right)^2\right] \le
C\begin{cases}
n, & \beta>1,
\\
n\log n, & \beta = 1,
\\
n^{2-\beta}, & 0<\beta<1.
\end{cases}
\eeqn
Notice that for any $\ve>0$ we have $n\log n= O(n^{1+\ve})$. Applying Theorem~\ref{thm:GK} with
\beqn
q=
 \begin{cases}
1, & \beta > 1
\\
1+ \delta, & \beta = 1
\\
2-\beta, & 0<\beta < 1
\end{cases}
\eeqn
yields the claim.
\end{proof}

\begin{prop}\label{prop:strong-mixing_GK}
Let $\eta(n)=Cn^{-\psi}$, $\psi>1$ and $\alpha(n)=Cn^{-\gamma}$, $\gamma>0$. Then, for any~$\delta>0$,
\beqn
\frac1n\sum_{k=1}^n \tilde{v}_k = 
\begin{cases}
O(n^{-\frac 12} \log^{\frac32+\delta}n), & \min\{\psi-1,\gamma\} > 1,
\\
O(n^{-\frac12+\delta}), & \min\{\psi-1,\gamma\} = 1,
\\
O(n^{-\frac{\min\{\psi-1,\gamma\}}{2}} \log^{\frac32+\delta}n), & 0<\min\{\psi-1,\gamma\}<1
\end{cases}
\eeqn
almost surely.
\end{prop}
\begin{proof}
 By Lemma \ref{lem:rbounds}, we have
 $|\bE[\tilde v_i\tilde v_k]| \le   Cm^{-\min\{\psi-1,\gamma\}}$. Applying Lemma~\ref{lem:tilv_sum2} with $\beta=\min\{\psi-1,\gamma\}$ yields the claim.
\end{proof}

We are now in position to prove the main result of this section:

\begin{thm}\label{thm: sigma_n^2}
Assume (SA1) and (SA3) with $\eta(n)=Cn^{-\psi}$, $\psi>1$. Assume (SA2) with $\alpha(n)=Cn^{-\gamma}$, $\gamma>0$.
Then, for arbitrary $\delta>0$,
\beqn
|\sigma_{n}^{2}-\bE\sigma_{n}^{2}| = 
\begin{cases}
O(n^{-\frac 12} \log^{\frac32+\delta}n), & \min\{\psi-1,\gamma\} > 1,
\\
O(n^{-\frac12+\delta}), & \min\{\psi-1,\gamma\} = 1,
\\
O(n^{-\frac{\min\{\psi-1,\gamma\}}{2}} \log^{\frac32+\delta}n), & 0<\min\{\psi-1,\gamma\}<1
\end{cases}
\eeqn
almost surely.
\end{thm}
\begin{proof}

By Corollary~\ref{cor:var_mean_poly},
\beqn
\left|\sigma_{n}^{2}-\bE\sigma_{n}^{2} - \frac{1}{n}\sum_{i=0}^{n-1}\tilde{v}_{i}\right| \le C 
\begin{cases}
n^{-1}, & \psi > 2,\\
n^{-1}\log n, & \psi = 2,\\
n^{1-\psi}, & 1<\psi<2.
\end{cases}
\eeqn
Combining this with Proposition~\ref{prop:strong-mixing_GK} yields the following upper bounds on $|\sigma_{n}^{2}-\bE\sigma_{n}^{2}|$:
\beqn
\begin{split}
\begin{cases}O(n^{-\frac 12} \log^{\frac32+\delta}n+n^{-1}), & \min\{\psi-1,\gamma\} > 1,
\\
O(n^{-\frac12+\delta}+n^{-1}\log n), & \min\{\psi-1,\gamma\} = 1,
\\
O(n^{-\frac{\min\{\psi-1,\gamma\}}{2}} \log^{\frac32+\delta}n+n^{-\min\{1,\psi-1\}}+n^{-1}\log n), & 0<\min\{\psi-1,\gamma\}<1.
\end{cases}
\end{split}
\eeqn
In each case the first term is the largest, so the proof is complete.
\end{proof}


\section{The term $|\bE\sigma_n^2 - \sigma^2|$}\label{sec:conv of sigma_n^2 to sigma^2}
In this section we formulate general condition that allow to identify the limit $\sigma^2 = \lim_{n\to\infty}\bE\sigma_n^2$ and obtain a rate of convergence. 

Write
\beqn
c_{ij} = \mu(f_i f_j) - \mu(f_i)\mu(f_j)
\eeqn
for brevity.
Then
\beqn
\sigma_n^2 = \frac1n\sum_{i=0}^{n-1}\sum_{j=0}^{n-1} c_{ij} 
= \frac1n\sum_{i=0}^{n-1} \sum_{j=i}^{n-1}(2-\delta_{ij})c_{ij} 
= \frac1n\sum_{i=0}^{n-1} \sum_{k=0}^{n-1-i}(2-\delta_{i,i+k})c_{i,i+k} .
\eeqn
Setting\footnote{The relationship with earlier notations is that $\sum_{k=0}^\infty v_{ik} = v_i$.}
\beqn
v_{ik} = (2-\delta_{k0})[\mu(f_i f_{i+k}) - \mu(f_i)\mu(f_{i+k})]
\eeqn
we arrive at
\beqn
\sigma_n^2 = \frac1n\sum_{i=0}^{n-1} \sum_{k=0}^{n-1-i} v_{ik}
\qquad\text{and}\qquad
\bE\sigma_n^{2} = \frac1n\sum_{i=0}^{n-1} \sum_{k=0}^{n-1-i} \bE v_{ik}.
\eeqn
Recall that
\beq\label{eq:SA_ok}
|v_{ik}| \le 2\eta(k).
\eeq

\subsection{Asymptotics of related double sums of real numbers}
In this subsection we consider double sequences of uniformly bounded numbers $a_{ik}$, $(i,k)\in\bN^2$, with the objective of controlling the sequence 
\beqn
B_n = \frac1n\sum_{i=0}^{n-1} \sum_{k=0}^{n-1-i} a_{ik}
\eeqn
for large values of~$n$. In this subsection, we make the following assumption tailored to our later needs:

\smallskip
\noindent{\bf Standing assumption.} There exists $\eta:\bN\to[0,\infty)$ such that
\beq\label{eq:standing}
|a_{ik}| \le \eta(k)
\qquad
\text{and}
\qquad
\sum_{k=0}^\infty\eta(k)<\infty.
\eeq
We also denote the tail sums of $\eta$ by
\beqn
R(K) = \sum_{k = K+1}^{\infty}\eta(k).
\eeqn

We begin with a handy observation:
\begin{lem}\label{lem:pre-mean}
There exists $C>0$ such that
\beqn
\left | B_n - \frac1n\sum_{i = 0}^{n-1}\sum_{k = 0}^{L} a_{ik} \right| \le C(R(K) + Kn^{-1})
\eeqn
whenever $0<K\le n$ and $K\le L\le\infty$.
\end{lem}

\begin{proof}
For all choices of $0<K\le n$ we have
\beqn
\begin{split}
B_n & 
= \frac1n\sum_{i = 0}^{n-K-1}\sum_{k = 0}^{n-1-i} a_{ik} + \frac1n\sum_{i = n-K}^{n-1}\sum_{k = 0}^{n-1-i} a_{ik} 
= \frac1n\sum_{i = 0}^{n-K-1}\sum_{k = 0}^{n-1-i} a_{ik} + O(Kn^{-1})
\\
& = \frac1n\sum_{i = 0}^{n-K-1}\sum_{k = 0}^{K} a_{ik} + \frac1n\sum_{i = 0}^{n-K-1}\sum_{k = K+1}^{n-1-i} a_{ik} + O(Kn^{-1})
\\
& = \frac1n\sum_{i = 0}^{n-K-1}\sum_{k = 0}^{K} a_{ik} + O(R(K) + Kn^{-1})
\\
& = \frac1n\sum_{i = 0}^{n-1}\sum_{k = 0}^{K} a_{ik} - \frac1n\sum_{i = n-K}^{n-1}\sum_{k = 0}^{K} a_{ik} + O(R(K) + Kn^{-1})
\\
& = \frac1n\sum_{i = 0}^{n-1}\sum_{k = 0}^{K} a_{ik} + O(R(K) + Kn^{-1}).
\end{split}
\eeqn
The error is uniform because of the uniform condition $|a_{ik}|\le\eta(k)$. For $L\ge K$,
\beqn
\sum_{k = 0}^{L} a_{ik} - \sum_{k = 0}^{K} a_{ik} = \sum_{k = K+1}^{L} a_{ik} = O(R(K))
\eeqn
uniformly, which concludes the proof.
\end{proof}

The following lemma helps identify the limit of $B_n$ and the rate of convergence under certain circumstances:
\begin{lem}\label{lem:mean}
Suppose the limit
\beqn
b_k = \lim_{n\to\infty} \frac1n\sum_{i = 0}^{n-1} a_{ik}
\eeqn
exists for all $k\ge 0$. Then
\beqn
\lim_{n\to\infty} B_n = \sum_{k=0}^\infty b_k.
\eeqn
The series on the right side converges absolutely. Furthermore, denoting
\beqn
r_k(n) =  \frac1n\sum_{i = 0}^{n-1} a_{ik} - b_k 
\eeqn
there exists $C>0$ such that
\beq\label{eq:mean-rate}
\left|B_n - \sum_{k=0}^\infty b_k\right| \le C\!\left(\left|\sum_{k = 0}^{K}r_k(n)\right| + R(K) + Kn^{-1}\right)
\eeq
holds whenever $0<K\le n$.
\end{lem}

\begin{proof}
Since
\beqn
\left|\frac1n\sum_{i = 0}^{n-1} a_{ik}\right| \le \eta(k),
\eeqn
also $|b_k|\le \eta(k)$, so the series $\sum_{k=0}^\infty b_k$ converges absolutely. 
Lemma~\ref{lem:pre-mean} with $L=K$ yields
\beqn
\begin{split}
B_n = \sum_{k = 0}^{K} \frac1n\sum_{i = 0}^{n-1} a_{ik} + O(R(K) + Kn^{-1})
\end{split}
\eeqn
uniformly for all $0<K\le n$. Thus, the definition of $r_k(n)$ gives
\beqn
\begin{split}
B_n 
& = \sum_{k = 0}^{K} b_{k} + O\!\left(\left|\sum_{k = 0}^{K}r_k(n)\right| + R(K) + Kn^{-1}\right).
\end{split}
\eeqn
Now $|b_k|\le\eta(k)$ yields~\eqref{eq:mean-rate}. To prove the convergence of $B_n$, consider~\eqref{eq:mean-rate} and fix an arbitrary~$\ve>0$. Fix~$K$ so large that~$R(K)<\ve/2C$. Since $\bigl|\sum_{k = 0}^{K}r_k(n)\bigr| + Kn^{-1}$ tends to zero with increasing $n$, it is bounded by $\ve/2C$ for all large $n$. Then $|B_n -  \sum_{k = 0}^{\infty} b_{k} | < \ve$.
\end{proof}

\subsection{Convergence of $\bE\sigma_n^2$: a general result}
In this subsection we apply the results of the preceding subsection to the sequence
\beqn
\bE\sigma_n^{2} = \frac1n\sum_{i=0}^{n-1} \sum_{k=0}^{n-1-i} \bE v_{ik}
\eeqn
where
\beqn
\bE v_{ik} = (2-\delta_{k0})\,\bE[\mu(f_i f_{i+k}) - \mu(f_i)\mu(f_{i+k})].
\eeqn
Recall from~\eqref{eq:SA_ok} and~\eqref{eq:weak_eta} of (SA1) that the Standing assumption in~\eqref{eq:standing} is satisfied:
$
|\bE v_{ik}| \le 2\eta(k)
$
and $\sum_{k=0}^\infty\eta(k)<\infty$.
The next theorem is nothing but a rephrasing of Lemma~\ref{lem:mean} in the case $a_{ik} = \bE v_{ik}$ at hand.
\begin{thm}\label{thm:var_limit}
Suppose the limit
\beqn
V_k = \lim_{n\to\infty} \frac1n\sum_{i = 0}^{n-1} \bE v_{ik}
\eeqn
exists for all $k\ge 0$. The series
\beqn
\sigma^2 = \sum_{k=0}^\infty V_k
\eeqn
is absolutely convergent, and
\beqn
\lim_{n\to\infty} \bE\sigma_n^{2} = \sigma^2.
\eeqn
In particular, $\sigma^2\ge 0$. Furthermore, there exists a constant $C>0$ such that
\beqn
|\bE\sigma_n^{2} - \sigma^2| \le C\!\left(\left|\,\sum_{k = 0}^{K}\!\left(\frac1n\sum_{i = 0}^{n-1} \bE v_{ik} - V_k \right)\right| + \sum_{k=K+1}^\infty\eta(k) + Kn^{-1}\right)
\eeqn
holds whenever $0<K\le n$.
\end{thm}

\subsection{Convergence of $\bE\sigma_n^2$: asymptotically mean stationary~$\bP$}
For the rest of the section we assume~$\bP$ is asymptotically mean stationary, with mean $\bar\bP$. In other words, there exists a measure~$\bar\bP$ such that, given a bounded measurable $g:\Omega\to\bR$,
\beq\label{eq:AMS}
\lim_{n\to\infty}\frac1n\sum_{i=0}^{n-1}\int g\circ \tau^i\,\rd\bP = \int g\,\rd\bar\bP.
\eeq
The measure $\bar\bP$ is then $\tau$-invariant. We denote $\bar\bE g=\int g\,\rd\bar\bP$. We will shortly impose additional rate conditions; see~\eqref{eq:AMSrate}.

Recall the cocycle property of the random compositions.  In what follows, it will be convenient to use the notations
\beqn
g_{ik}^1(\omega) = \mu(f_i f_{i+k}) = \mu(f\circ\varphi(i,\omega) f\circ\varphi(i+k,\omega))
\eeqn
and
\beqn
g_{ik}^2(\omega) = \mu(f_i) \mu(f_{i+k}) = \mu(f\circ\varphi(i,\omega))\mu(f\circ\varphi(i+k,\omega)).
\eeqn

For the results of this section we need the following preliminary lemma, which crucially relies on the memory-loss property (SA3), assumed to hold throughout this text.

\begin{lem}\label{lem:pre}
There exists a constant $C>0$ such that
\beq\label{eq:pre}
|g_{ik}^a - g_{rk}^a\circ\tau^{i-r}| \le C\eta(r)
\eeq
for all $0\le r\le i$, $k\ge 0$ and $a\in\{1,2\}$.
\end{lem}
\begin{proof}
Note that we may rewrite the memory-loss property in~\eqref{eq:memloss_cocycle} as
\beqn
|\mu(g\circ\varphi(j,\omega))- \mu(g\circ\varphi(r,\tau^{j-r}\omega))| \le C\eta(r),
\eeqn
for all $r\le j$.
Thus, choosing $g=f$ (recall $f = f_{j,j+1}\in\cG_j$ for all $j$) yields
\beqn
\begin{split}
& \ |g_{ik}^2 - g_{rk}^2\circ\tau^{i-r}|\\
 = & \ |\mu(f\circ\varphi(i,\omega))\mu(f\circ\varphi(i+k,\omega))- \mu(f\circ\varphi(r,\tau^{i-r}\omega))\mu(f\circ\varphi(r+k,\tau^{i-r}\omega)) |
 \\
 \le & \ |\mu(f\circ\varphi(i,\omega))| |\mu(f\circ\varphi(i+k,\omega))- \mu(f\circ\varphi(r+k,\tau^{i-r}\omega))|
 \\
& \quad +  \ |\mu(f\circ\varphi(i,\omega))- \mu(f\circ\varphi(r,\tau^{i-r}\omega))| |\mu(f\circ\varphi(r+k,\tau^{i-r}\omega))|
\\
\le  & \ C(\eta(r+k)+\eta(r))\le C\eta(r).
\end{split}
\eeqn
On the other hand, choosing $g=ff_{i+k,i+1}=f f\circ \varphi(k,\tau^{i}\omega)$ gives
\beqn
\begin{split}
& \ |g_{ik}^1 - g_{rk}^1\circ\tau^{i-r}|
\\
 = & \ |\mu(f\circ\varphi(i,\omega) f\circ\varphi(i+k,\omega))- \mu(f\circ\varphi(r,\tau^{i-r}\omega) f\circ\varphi(r+k,\tau^{i-r}\omega)) |
\\
=& \ |\mu(g\circ \varphi(i,\omega))-\mu(g\circ\varphi(r,\tau^{i-r}\omega))|
\le   C\eta(r),
 \end{split}
\eeqn
which completes the proof.
\end{proof}

The following lemma guarantees that both limits $\lim_{n\rightarrow \infty}n^{-1}\sum_{i=0}^{n-1} \bE \mu(f_i f_{i+k})$ and $\lim_{n\rightarrow \infty}n^{-1}\sum_{i=0}^{n-1} \bE \mu(f_i)\mu(f_{i+k})$ exist and can be expressed in terms of $\bar{\bP}$.

\begin{lem} \label{lem: bar E limits}
For all $i, k\ge 0$ and $a\in\{1,2\}$
\beqn
\lim_{n\rightarrow\infty} n^{-1}\sum_{i=0}^{n-1}\bE g_{ik}^a=\lim_{j\rightarrow \infty} \bar{\bE}g_{jk}^{a}.
\eeqn
In particular, the limits exist.
\end{lem}
\begin{proof} First we make the observation that since $\bar \bP$ is stationary, \eqref{eq:pre} implies
\beqn
|\bar\bE g_{ik}^a - \bar\bE g_{rk}^a| \le C\eta(r)
\eeqn
whenever $i\ge r$. From assumption \eqref{eq:weak_eta} it follows that $\lim_{r\to\infty}\eta(r)=0$.
The sequence $(\bar\bE g_{ik}^a)_{i=0}^\infty$ is therefore Cauchy, so $\lim_{i\rightarrow\infty} \bar{\bE}g_{ik}^a$ exists and respects the same bound, i.e., 
\beq\label{eq:bar E conv}
|\bar{\bE} g_{rk}^a - \lim_{i\rightarrow\infty} \bar{\bE}g_{ik}^a| \le C\eta(r).
\eeq
 We are now ready to show that $\lim_{n\rightarrow\infty} n^{-1}\sum_{i=0}^{n-1}\bE g_{ik}^a$ exists and in the process we see that it is equal to $\lim_{j\rightarrow \infty} \bar{\bE}g_{jk}^{a}$.

Let $\ve>0$. Choose $r\in \bN$ such that $C\eta(r)< \ve/5$, where $C$ is the same constant as above. Then choose $n_{0}\in \bN$ that satisfies two following conditions. First, ~$\|f\|_{\infty}^{2}r/n_{0}< \ve/5$. Second, by~\eqref{eq:AMS}, $|n^{-1}\sum_{i=0}^{n-1}\bE g_{rk}^{a}\circ \tau^{i}-\bar{\bE}g_{rk}^{a}|< \ve/5$ for all $n\ge n_{0}$. Next we show that $|n^{-1}\sum_{i=0}^{n-1}\bE g_{ik}^{a}- \lim_{j\rightarrow\infty}\bar{\bE}g_{ik}^{a}|<\ve$  for all~$n\ge n_{0}$.

The following five estimates yield the desired result:

In this first estimate, note that $\|g_{ik}^{a}\|_{\infty}\le \|f\|_{\infty}^{2}$ for all $i,k\in \bN$ and $a\in \{1,2\}$: 
\beqn
\left|\frac{1}{n} \sum_{i=0}^{n-1}\bE g_{ik}^{a}- \frac{1}{n} \sum_{i=r}^{n-1}\bE g_{ik}^{a}\right|\le \|f^{2}\|_{\infty}n^{-1}r < \frac{\ve}{5}.
\eeqn
In the second estimate, we apply \eqref{eq:pre}:
\beqn
\left|\frac{1}{n} \sum_{i=r}^{n-1}\bE g_{ik}^{a}- \frac{1}{n} \sum_{i=0}^{n-r-1}\bE g_{rk}^{a}\circ \tau^{i}\right|= \left|\frac{1}{n} \sum_{i=r}^{n-1}\bE g_{ik}^{a}- \frac{1}{n} \sum_{i=r}^{n-1}\bE g_{rk}^{a}\circ \tau^{i-r}\right|\le C\eta(r)< \frac{\ve}{5}.
\eeqn
The third estimate follows the same reasoning as the first:
\beqn
\left| \frac{1}{n} \sum_{i=0}^{n-r-1}\bE g_{rk}^{a}\circ \tau^{i} -  \frac{1}{n} \sum_{i=0}^{n-1}\bE g_{rk}^{a}\circ \tau^{i} \right|\le \|f\|_{\infty}^{2}n^{-1}r < \frac{\ve}{5}.
\eeqn
The fourth estimate follows by the definition of $n_{0}$:
\beqn
\left| \frac{1}{n} \sum_{i=0}^{n-1}\bE g_{rk}^{a}\circ \tau^{i} -  \bar{\bE}g_{rk}^{a} \right|  < \frac{\ve}{5}.
\eeqn
The last estimate holds by \eqref{eq:bar E conv}:
\beqn
\left| \bar{\bE}g_{rk}^{a} - \lim_{j\rightarrow \infty} \bar{\bE}g_{jk}^{a} \right| \le C\eta(r)< \frac{\ve}{5}.
\eeqn

These estimates combined, yield $|n^{-1}\sum_{i=0}^{n-1}\bE g_{ik}^a -\lim_{j\rightarrow \infty} \bar{\bE}g_{jk}^{a}| < \ve$ for all $n\ge n_{0}$. Since ~$\lim_{j\rightarrow \infty} \bar{\bE}g_{jk}^{a}$ exists, then also $\lim_{i\rightarrow \infty}n^{-1}\sum_{i=0}^{n-1}\bE g_{ik}^a$ exists and is equal to it. 
\end{proof}

Theorem \ref{thm:var_limit} yields the next result as a corollary.
\begin{thm}\label{thm:sigma2n conv}
The series
\beqn
\sigma^2 = \sum_{k=0}^\infty V_k,
\eeqn
where
\beqn
V_{k}=(2-\delta_{k0})\lim_{r\to\infty}    \bar{\bE} [\mu(f_r f_{r+k}) - \mu(f_r)\mu(f_{r+k})],
\eeqn
is absolutely convergent, 
and 
\beqn
\lim_{n\to\infty}
\bE\sigma_n^{2} = \sigma^2.
\eeqn
\end{thm}

\begin{proof}
Recall that $\bE v_{ik} = (2-\delta_{k0})\,\bE[\mu(f_i f_{i+k}) - \mu(f_i)\mu(f_{i+k})]$.
By Lemma \ref{lem: bar E limits} the limits $\lim_{n\rightarrow\infty} n^{-1}\sum_{i=0}^{n-1}\bE  \mu(f_i f_{i+k})=\lim_{j\rightarrow \infty} \bar{\bE} \mu(f_j f_{j+k})$  and $\lim_{n\rightarrow\infty} n^{-1}\sum_{i=0}^{n-1}\bE  \mu(f_i) \mu(f_{i+k})=\lim_{j\rightarrow \infty} \bar{\bE}\mu(f_j) \mu(f_{j+k})$ exist. Therefore 
\beqn
V_k = \lim_{n\to\infty} \frac1n\sum_{i = 0}^{n-1} \bE v_{ik}= (2-\delta_{k0})\lim_{r\to\infty}    \bar{\bE} [\mu(f_r f_{r+k}) - \mu(f_r)\mu(f_{r+k})].
\eeqn
Now the rest of the claim follows from Theorem \ref{thm:var_limit}.
\end{proof}

\medskip
\noindent{\bf Standing Assumption (SA4).} For the rest of the paper we assume that $\bP$ is asymptotically mean stationary, and there exist $C_0>0$ and $\zeta>0$ such that
\beq\label{eq:AMSrate}
\sup_{r,k,a}\left|\frac1n\sum_{i=0}^{n-1}\int g_{rk}^a\circ \tau^i\,\rd\bP -\int g_{rk}^a\,\rd\bar\bP\right|\le C_0n^{-\zeta}
\eeq
for all $n\ge 1$.\hfill$\blacksquare$
\medskip

\begin{lem}\label{lem: n1 to n2}
For all integers $0<n_{1}< n_{2}$,
\beqn
\left|(n_{2}-n_{1})^{-1}\sum_{i=n_{1}}^{n_{2}-1}\bE g_{ik}^a- \lim_{r\rightarrow \infty} \bar{\bE} g_{rk}^a\right|\le C(\eta(n_{1})+(n_{2}-n_{1})^{-\zeta}),
\eeqn
where $C$ is uniform.
\end{lem}
\begin{proof}
By Lemma \ref{lem:pre} we have
\beq
\begin{split}
 & \left|(n_{2}-n_{1})^{-1}\sum_{i=n_{1}}^{n_{2}-1}\bE g_{ik}^{a}- (n_{2}-n_{1})^{-1}\sum_{i=n_{1}}^{n_{2}-1}\bE [g_{n_{1}k}^{a}\circ \tau^{i-n_{1}}] \right|
\\ 
 \le & (n_{2}-n_{1})^{-1}\sum_{i=n_{1}}^{n_{2}-1} C\eta(n_{1})= C\eta(n_{1}). \label{eq:ams1}
\end{split}
\eeq
By \eqref{eq:AMSrate}, it follows that
\beq
\begin{split}
\left| (n_{2}-n_{1})^{-1}\sum_{i=n_{1}}^{n_{2}-1}\bE [g_{n_{1}k}^{a}\circ \tau^{i-n_{1}}]-\bar{\bE}g_{n_{1}k}^{a} \right| & =\left| (n_{2}-n_{1})^{-1}\sum_{i=0}^{n_{2}-n_{1}-1}\bE [g_{n_{1}k}^{a}\circ \tau^{i}]-\bar{\bE}g_{n_{1}k}^{a} \right|
 \\
 & \le C_{0}(n_{2}-n_{1})^{-\zeta}\label{eq:ams2}
\end{split}
\eeq 
Finally \eqref{eq:bar E conv} gives
\beq
|\bar{\bE} g_{n_{1}k}^{a}-\lim_{r\rightarrow \infty} \bar{\bE} g_{rk}^{a}|\le C\eta(n_{1}). \label{eq:ams3}
\eeq
Now the claim follows from \eqref{eq:ams1}, \eqref{eq:ams2} and \eqref{eq:ams3}.
\end{proof}

Next we use Lemma \ref{lem: n1 to n2} to provide an upper bound on $\left|n^{-1}\sum_{i=0}^{n-1}\bE g_{ik}^{a}- \lim_{r\rightarrow \infty} \bar{\bE} g_{rk}^{a}\right|$. Note that just making the substitutions $n_{1}=0$ and $n_{2}=n$ in Lemma \ref{lem: n1 to n2} does not yield a good result. Instead we divide the sum $\sum_{i=0}^{n-1}\bE g_{ik}^{a}$ into an increasing number of partial sums and then apply Lemma \ref{lem: n1 to n2} separately to those parts.

Before proceeding to the next lemma, we define a function $h_{\zeta}\colon \bN \rightarrow \bR$ which depends on the parameter $\zeta$ in the following way
\beq\label{eq:hzeta}
h_{\zeta}(n) = 
\begin{cases}
n^{-1}, & \zeta > 1,\\
n^{-1}\log n, & \zeta = 1,\\
n^{-\zeta}, & 0<\zeta<1.
\end{cases}
\eeq

\begin{lem}\label{lem: split sum}
Suppose $\eta(n)=Cn^{-\psi}$, $\psi>1$. Then a uniform bound
\beqn
\left|\frac{1}{n}\sum_{i=0}^{n-1}\bE g_{ik}^a-\lim_{r\rightarrow \infty} \bar{\bE} g_{rk}^a\right|\le  C h_{\zeta}(n)
\eeqn
holds.
\end{lem}
\begin{proof}
Denote $n^{*}=\lfloor\log_{2} n \rfloor$. We split the sum $\frac{1}{n}\sum_{i=0}^{n-1}\bE g_{ik}^{a}$ as
\beqn
\frac{1}{n}\sum_{i=0}^{n-1}\bE g_{ik}^{a}= \frac{1}{n}\bE g_{0k}^{a}+ \frac{1}{n}\sum_{j=0}^{n^{*}-1} \sum_{i=2^{j}}^{2^{j+1}-1}\bE g_{ik}^{a}+ \frac{1}{n}\sum_{i=2^{n^{*}}}^{n-1}\bE g_{ik}^{a}.
\eeqn
Obviously 
\beq
\left|\frac{1}{n}\bE g_{0k}^{a} - \frac{1}{n} \lim_{r\rightarrow \infty} \bar{\bE} g_{rk}^{a}\right|\le Cn^{-1}\label{eq:part1}
\eeq
Lemma \ref{lem: n1 to n2} yields 
\beqn
\left|\sum_{i=2^{j}}^{2^{j+1}-1}\bE g_{ik}^{a}- (2^{j+1}-2^{j})\lim_{r\rightarrow \infty} \bar{\bE} g_{rk}^{a}\right|\le 2^{j}C((2^{j})^{-\psi}+(2^{j+1}-2^{j})^{-\zeta})\le C( 2^{j(1-\psi)}+2^{j(1-\zeta)}).
\eeqn
Therefore
\beq\label{eq:part2}
\begin{split}
& \left|\frac{1}{n}\sum_{j=0}^{n^{*}-1} \sum_{i=2^{j}}^{2^{j+1}-1}\bE g_{ik}^{a} - \frac{1}{n}(2^{n^{*}}-1)\lim_{r\rightarrow \infty} \bar{\bE} g_{rk}^{a}\right|
\\
& = \frac{1}{n}\left| \sum_{j=0}^{n^{*}-1} \left(\sum_{i=2^{j}}^{2^{j+1}-1}\bE g_{ik}^{a}- (2^{j+1}-2^{j})\lim_{r\rightarrow \infty} \bar{\bE} g_{rk}^{a}\right)\right|
\\
& \le Cn^{-1} \sum_{j=0}^{n^{*}-1}( 2^{j(1-\psi)}+2^{j(1-\zeta)})\le C(n^{-1}+h_{\zeta}(n))\le C h_{\zeta}(n).
\end{split}
\eeq
Lemma \ref{lem: n1 to n2} also gives
\beq\label{eq:part3}
\begin{split}
& \left|\frac{1}{n}\sum_{i=2^{n^{*}}}^{n-1}\bE g_{ik}^{a}-\frac{1}{n}
(n-2^{n^{*}})\lim_{r\rightarrow \infty} \bar{\bE} g_{rk}^{a}\right|
\\
&= 
n^{-1}(n-2^{n^{*}})\left|(n-2^{n^{*}})^{-1}\sum_{i=2^{n^{*}}}
^{n-1}\bE g_{ik}^{a}-\lim_{r\rightarrow \infty} \bar{\bE} g_{rk}^{a}\right|
\\
& \le n^{-1}(n-2^{n^{*}})C((2^{n^{*}})^{-\psi}+(n-2^{n^{*}})^{-
\zeta})\le C(n^{-1}+h_{\zeta}(n))\le C h_{\zeta}(n).
\end{split}
\eeq
In the last line we used the fact that $n/2 \le 2^{n^{*}}\le n$, implying $n-2^{n^{*}}\le n/2$.
Collecting the estimates \eqref{eq:part1}, \eqref{eq:part2} and \eqref{eq:part3}, we deduce $|\frac{1}{n}\sum_{i=0}^{n-1}\bE g_{ik}^{a}-\lim_{r\rightarrow \infty} \bar{\bE} g_{rk}^{a}|\le  C h_{\zeta}(n)$.
\end{proof}

We are finally ready to state and prove the main result of this section:

\begin{thm}\label{thm: sigma_n^2 to sigma^2}
Assume (SA1) and (SA3) with $\eta(n)=Cn^{-\psi}$, $\psi>1$. Assume (SA4) with~$\zeta>0$. Then
\beqn
|\bE\sigma_n^{2} - \sigma^2| \le C 
\begin{cases}
n^{\frac{1}{\psi}-1}, & \zeta > 1,\\
(n\log^{-1} n)^{\frac{1}{\psi}-1}, & \zeta = 1,\\
n^{\frac{\zeta}{\psi}-\zeta}, & 0<\zeta<1.
\end{cases}
\eeqn
Here $\sigma^2$ is the quantity appearing in Theorem~\ref{thm:sigma2n conv}.
\end{thm}
\begin{proof}
Let $k\ge 0$. The previous lemma applied to case $a=1$ yields
\beq
\left|\frac{1}{n}\sum_{i=0}^{n-1}\bE[\mu(f_i f_{i+k})]-\lim_{r\rightarrow\infty}\bar{\bE}[\mu(f_{r}f_{r+k})]\right| \le  C h_{\zeta}(n).\label{eq:muff}
\eeq
Similarly in the case $a=2$
\beq
\left|\frac{1}{n}\sum_{i=0}^{n-1}\bE [\mu(f_i) (f_{i+k})]-\lim_{r\rightarrow\infty}\bar{\bE}[\mu(f_{r})\mu(f_{r+k})]\right|\le  C h_{\zeta}(n).\label{eq:mufmuf}
\eeq 
Equations \eqref{eq:muff}, \eqref{eq:mufmuf} and Theorem \ref{thm:sigma2n conv} imply that
\beqn
\begin{split}
& \left|V_k- \frac{1}{n}\sum_{i=0}^{n-1} (2-\zeta_{k0})\bE[\mu (f_i f_{i+k})-\mu(f_i)\mu(f_{i+k})]\right|
\\
 & \le  C\left|\lim_{r\rightarrow\infty} \bar{\bE}[ \mu(f_{r} f_{r+k})-\mu(f_r) \mu(f_{r+k})]- \frac{1}{n}\sum_{i=0}^{n-1} \bE[ \mu(f_i f_{i+k})-\mu(f_i)\mu(f_{i+k})]\right|
\\
& \le C h_{\zeta}(n).
\end{split}
\eeqn
We apply Theorem \ref{thm:var_limit}, which yields
\beq\label{eq:var_limit-rate}
\begin{split}
|\bE\sigma_n^{2} - \sigma^2| & \le C\!\left(\left|\,\sum_{k = 0}^{K}\! h_{\zeta}(n) \right| + \sum_{k=K+1}^\infty k^{-\psi} + Kn^{-1}\right)\le CK(h_{\zeta}(n)+K^{-\psi}),
\end{split}
\eeq
for all $0<K\le n$. The estimate on the right side of \eqref{eq:var_limit-rate} is minimized, when $h_{\zeta}(n)=K^{-\psi}$.
Therefore choosing
\beqn
K \asymp
\begin{cases}
n^{\frac{1}{\psi}}, & \zeta > 1,\\
(n\log^{-1} n)^{\frac{1}{\psi}}, & \zeta = 1,\\
n^{\frac{\zeta}{\psi}}, & 0<\zeta<1
\end{cases}
\eeqn
in \eqref{eq:var_limit-rate} completes the proof. 
\end{proof}

\section{Conclusions}\label{sec:main}
\subsection{Main result and consequences}
Theorems~\ref{thm: sigma_n^2} and~\ref{thm: sigma_n^2 to sigma^2} immediately yield the main result of the paper, given next. The bounds shown are elementary combinations of these theorems, so we leave the details to the reader. Let us remind the reader of the Standing Assumptions (SA1)--(SA4) in Sections~\ref{sec:conv of sigma_n^2} and~\ref{sec:conv of sigma_n^2 to sigma^2}. At the end of the section we also comment on the case of vector-valued observables.
\begin{thm}\label{thm:main}
Assume (SA1\&3) with $\eta(n)=Cn^{-\psi}$, $\psi>1$; (SA2) with $\alpha(n)=Cn^{-\gamma}$, $\gamma>0$; and (SA4) with~$\zeta>0$. Fix an arbitrarily small $\delta>0$.
Then there exists $\Omega_*\subset\Omega$, $\bP(\Omega_*)=1$, such that all of the following holds: The non-random number
\beqn
\sigma^2 = \sum_{k=0}^\infty (2-\delta_{k0})\lim_{i\to\infty}    \bar{\bE} [\mu(f_i f_{i+k}) - \mu(f_i)\mu(f_{i+k})]
\eeqn
is well defined, nonnegative, the series is absolutely convergent, and 
\beqn
\lim_{n\to\infty}\sigma_n^2(\omega) = \sigma^2
\eeqn
for every~$\omega\in\Omega_*$. Moreover, the absolute difference   
\beqn
\Delta_n(\omega) = |\sigma_{n}^{2}(\omega)-\sigma^{2}|
\eeqn
has the following upper bounds, for any $\omega\in\Omega_*$:
\beqn
\Delta_n = 
\begin{cases}
O(n^{-\frac 12} \log^{\frac32+\delta}n), & \zeta \ge 1, \quad\min\{\psi-1,\gamma\} > 1,
 \\
  O(n^{-\frac12+\delta}), & \zeta \ge 1,  \quad \min\{\psi-1,\gamma\} = 1,
  \\
  O(n^{-\frac{\min\{\psi-1,\gamma\}}{2}} \log^{\frac32+\delta}n), & \zeta \ge 1,  \quad 0<\min\{\psi-1,\gamma\}<1,
  \\
O(n^{\frac{\zeta}{\psi}-\zeta}+n^{-\frac 12} \log^{\frac32+\delta}n), &  0<\zeta<1, \quad\min\{\psi-1,\gamma\} > 1,
 \\
  O(n^{\frac{\zeta}{\psi}-\zeta}+ n^{-\frac12+\delta}), &  0<\zeta<1,  \quad \min\{\psi-1,\gamma\} = 1,
  \\
  O(n^{\frac{\zeta}{\psi}-\zeta}+ n^{-\frac{\min\{\psi-1,\gamma\}}{2}} \log^{\frac32+\delta}n), &  0<\zeta<1,  \quad 0<\min\{\psi-1,\gamma\}<1.
  \\
\end{cases}
\eeqn
\end{thm}

Let us reiterate that Theorem~\ref{thm:main} facilitates proving quenched central limit theorems with convergence rates for the fiberwise centered $\bar W_n$. Recalling the discussion from the beginning of the paper, we namely have the following trivial lemma (thus presented without proof):
\begin{lem}
Suppose~$d(\slot,\slot)$ is a distance of probability distributions with the following property: given~$b>0$, there exist an open neighborhood~$U\subset\bR_+$ of~$b$ and a constant~\mbox{$C>0$}, such that
\beq\label{eq:Gaussian_distance}
d(a Z,b Z) \le C|a - b|
\eeq
for all $a\in U$. Here $Z\sim\cN(0,1)$. If $\sigma^2>0$, then for every $\omega\in\Omega_*$,
\beqn
d(\bar W_n,\sigma Z) \le d(\bar W_n,\sigma_n Z) + O(\Delta_n).
\eeqn

\end{lem}
In other words, once a bound on the first term on the right side has been established (e.g., using methods cited earlier), one can use Theorem~\ref{thm:main} to bound the second term almost surely. Typical metrics satisfying~\eqref{eq:Gaussian_distance} are the 1-Lipschitz (Wasserstein) and Kolmogorov distances. 

The results presented above allow to formulate some sufficient conditions for $\sigma^2 > 0$. For simplicity, we proceed in the ideal parameter regime
\beq\label{eq:best_parameters}
\min\{\psi-1,\gamma,\zeta\}>1.
\eeq
Generalizations of the next result involving any of the other parameter regimes of Theorem~\ref{thm:main} are straightforward, and left to the reader.
\begin{cor}\label{cor:nonzero_variance}
Let~\eqref{eq:best_parameters} hold with all the other assumptions of Theorem~\ref{thm:main}. Suppose that either (i) there exists $\omega\in\Omega_*$ such that
\beqn
\sup_{n\ge 2}\frac{\Var_\mu(S_n)}{n^{\frac 12} \log^{\frac32+\delta}n} = \infty
\eeqn
or (ii) and
\beqn
\sup_{n\ge 1}\frac{\bE\Var_\mu(S_n)}{n^{\frac1{\psi}}} = \infty.
\eeqn
Then $\sigma^2 > 0$.
\end{cor}

\begin{proof}
Suppose $\sigma^2 = 0$. We will derive a contradiction in each case.

\noindent (i) Let $\omega\in\Omega_*$ be arbitrary. By Theorem~\ref{thm:main}, there exists $C>0$ such that $n^{-1}\Var_\mu(S_n) = \sigma_n^2 \le C n^{-\frac 12} \log^{\frac32+\delta}n$ for all $n\ge 1$. This violates the assumption of part~(i), so $\sigma^2>0$.

\noindent (ii) By Theorem~\ref{thm: sigma_n^2 to sigma^2}, there exists $C>0$ such that $n^{-1}\bE\Var_\mu(S_n) = \bE\sigma_n^2 \le Cn^{\frac{1}{\psi}-1}$ for all $n\ge 1$. This violates the assumption of part~(ii), so $\sigma^2>0$.
\end{proof}

We will return to the question of whether $\sigma^2 = 0$ or $\sigma^2>0$ in Lemma~\ref{lem:zerovariance}.

\subsection{Vector-valued observables}
Let us conclude by explaining, as promised, how the results extend with ease to the case of a vector-valued observable $f:X\to\bR^d$. This time $\sigma_n^2$ is a $d\times d$ covariance matrix and, if the limit exists, so is $\sigma^2 = \lim_{n\to\infty}\sigma_n^2$. Define the functions $\ell_n:\bR^d\to\bR$ by
\beqn
\ell_n(v) = v^T\sigma_n^2 v,
\eeqn
and denote the standard basis vectors of~$\bR^d$ by $e_\alpha$, $\alpha = 1,\dots,d$.  Observe that $\ell_n(v)$ is the $\mu$-variance of $W_n$ with the scalar-valued observable $v^T f$ in place of $f$.

\begin{lem}
Suppose there exists~$\kappa>0$ such that, almost surely, the limit $\ell(e_\alpha+e_\beta) = \lim_{n\to\infty}\ell_n(e_\alpha+e_\beta)$ exists and
\beqn
\ell(e_\alpha+e_\beta) - \ell_n(e_\alpha+e_\beta) = O(n^{-\kappa})
\eeqn
as $n\to\infty$ for all $\alpha,\beta=1,\dots,d$. Then, almost surely, $\sigma^2 = \lim_{n\to\infty}\sigma_n^2$ exists and
\beqn
|\sigma^2 - \sigma_n^2| = O(n^{-\kappa})
\eeqn
for all matrix norms.
\end{lem}
\begin{proof}
Note that the matrix elements of~$\sigma_n^2$ are given by
\beqn
(\sigma_n^2)_{\alpha\beta} = \tfrac12(\ell_n(e_\alpha+e_\beta) - \ell_n(e_\alpha) - \ell_n(e_\beta)).
\eeqn
Dropping the subindex $n$ yields the limit matrix elements~$\sigma^2_{\alpha\beta}$. Since~$\alpha$ and~$\beta$ can take only finitely many values, simultaneous almost sure convergence for the matrix elements with the claimed rate follows.
\end{proof}

According to the lemma, the rate of convergence of the covariance matrix~$\sigma_n^2$ to~$\sigma^2$ can be established by applying the earlier results to the finite family of scalar-valued observables $(e_\alpha + e_\beta)^T f$. Further, one may apply Corollary~\ref{cor:nonzero_variance} (or Lemma~\ref{lem:zerovariance}) to the observables $v^T f$ for all unit vectors~$v$ to obtain conditions for~$\sigma^2$ being positive definite. Assuming now it is, for certain metrics (e.g.\@ 1-Lipschitz) one has
\beqn
d(\sigma_n Z,\sigma Z) \le C|\sigma^2-\sigma_n^2|
\eeqn
where $Z\sim\cN(0,I_{d\times d})$ and $C=C(\sigma)$, which again yields an estimate of the type
\beqn
d(\bar W_n,\sigma Z) \le d(\bar W_n,\sigma_n Z) + C|\sigma^2-\sigma_n^2|.
\eeqn
We refer the reader to Hella~\cite{Hella_2018} for details, including the hard part of establishing an almost sure, asymptotically decaying bound on~$d(\bar W_n,\sigma_n Z)$ in the vector-valued case.


\appendix

\section{Random dynamical systems}\label{sec:RDS}
In this section we interpret the limit variance of Theorems~\ref{thm:sigma2n conv} and~\ref{thm: sigma_n^2 to sigma^2} from the point of view of random dynamical systems. Like elsewhere in the paper, we will assume the system possesses the good, uniform, fiberwise properties of the Standing Assumptions.

Recall that $\tau$ preserves the probability measure~$\bar\bP$ in~\eqref{eq:AMS}, i.e., $\tau^{-1} F\in\cF$ and $\bar\bP(\tau^{-1} F) = \bar\bP(F)$ for all $F\in\cF$. One says that $\varphi(\slot,\slot,\slot)$ in~\eqref{eq:varphi} is a measurable random dynamical system (RDS) on the measurable space $(X,\cB)$ over the measure-preserving transformation $(\Omega,\cF,\bar\bP,\tau)$. The map
\beqn
\Phi:\Omega\times X\to \Omega\times X:\Phi(\omega,x) = (\tau\omega,\varphi(1,\omega)x) = (\tau\omega, T_{\omega_1}(x))
\eeqn
is called the skew product of the measure-preserving transformation $(\Omega,\cF,\bar\bP,\tau)$ and the cocycle $\varphi(n,\omega)$ on~$X$. It is a measurable self-map on $(\Omega\times X,\cF\otimes\cB)$. In general, random dynamical systems and skew products have one-to-one correspondence; in particular, the measurability of one implies the measurability of the other.

We are interested in the skew product, because of the identity
\beqn
\Phi^n(\omega,x) = (\tau^n\omega,\varphi(n,\omega)x) = (\tau^n\omega,T_{\omega_n}\circ\dots\circ T_{\omega_1}(x)).
\eeqn
Thus, our task is to study the statistics of the projection of $\Phi^n(\omega,x)$ to~$X$. It now becomes interesting to study the invariant measures of~$\Phi$. However, the class of all invariant measures of~$\Phi$ is unnatural, for we must incorporate the fact that $\tau$ preserves the measure~$\bar\bP$. For this reason, it is said that a probability measure~$\bfP$ on~$\cF\otimes\cB$ is an invariant measure for the RDS~$\varphi$ if it is invariant for $\Phi$ \emph{and} the marginal of $\bfP$ on $\Omega$ coincides with $\bar\bP$. In other words,
\beqn
\Phi_*\bfP = \bfP
\quad\text{and}\quad
(\Pi_1)_*\bfP = \bar\bP,
\eeqn
where $\Pi_1:\Omega\times X\to X:(\omega,x)\mapsto \omega$.

We will also need to consider the cocycle
\beqn
\varphi^{(2)}(n,\omega)(x,y) = (\varphi(n,\omega)x,\varphi(n,\omega)y)
\eeqn
on the product space $X\times X$. The corresponding skew product is
\beqn
\Phi^{(2)}(\omega,x,y) = (\tau\omega,\varphi^{(2)}(1,\omega)(x,y)). 
\eeqn
The invariant measures of the RDS~$\varphi^{(2)}$ are defined analogously to above. Without danger of confusion, we define the projections $\Pi_1(\omega,x,y) = \omega$, $\Pi_2(\omega,x,y) = x$ and $\Pi_3(\omega,x,y) = y$ on $\Omega\times X\times X$. We also write $(\Pi_1\times \Pi_2)(\omega,x,y) = (\omega,x)$ and $(\Pi_1\times \Pi_3)(\omega,x,y) = (\omega,y)$.

Of particular interest will be the sequence of functions $Z_n:\Omega\times X\times X$ defined by
\beqn
Z_n(\omega,x,y) = S_n(\omega,x)-S_n(\omega,y).
\eeqn
For then
\beq\label{eq:varZ}
\begin{split}
& \int Z_n^2(\omega)\,\rd(\mu\otimes\mu) 
= 2 \Var_\mu(S_n(\omega)) = 2\sigma_n^2(\omega)\cdot n.
\end{split}
\eeq
Notice already that writing
\beqn
F(\omega,x,y) = f(x) - f(y)
\eeqn
yields the identity
\beq\label{eq:Zn_sum}
Z_n = \sum_{i=0}^{n-1} F\circ(\Phi^{(2)})^i.
\eeq

\medskip
\noindent{\bf Standing Assumption (SA5).} 
Assume there exists an invariant measure~$\bfP^{(2)}$ for the RDS $\varphi^{(2)}$ that is symmetric in the sense that
\beq\label{eq:P^2_symmetry}
\int h(\omega,x,y)\,\rd\bfP^{(2)}(\omega,x,y) = \int h(\omega,y,x)\,\rd\bfP^{(2)}(\omega,x,y)
\eeq
for all bounded measurable $h:\Omega\times X\times X\to\bR$. 
The common marginal
\beq\label{eq:marginals_P}
\bfP = (\Pi_1\times\Pi_2)_* \bfP^{(2)} = (\Pi_1\times\Pi_3)_* \bfP^{(2)}
\eeq
is then trivially an invariant measure for the RDS $\varphi$. Moreover, assume
\beq\label{eq:SA5_lim1}
\lim_{i\to\infty}    \bar{\bE} [\mu(f_i)] = \int f(x)\,\rd\bfP(\omega,x),
\eeq
\beq\label{eq:SA5_lim2}
\lim_{i\to\infty}    \bar{\bE} [\mu(f_i f_{i+k})] = \int f(x)\,f(\varphi(k,\omega,x))\,\rd\bfP(\omega,x)
\eeq
and
\beq\label{eq:SA5_lim3}
\lim_{i\to\infty}    \bar{\bE} [\mu(f_i)\mu(f_{i+k})] = \int f(x)\, f(\varphi(k,\omega,y))\,\rd\bfP^{(2)}(\omega,x,y)
\eeq
are satisfied.\hfill$\blacksquare$
\medskip

While Standing Assumption (SA5) may, from the point of view of the initial setup of our problem, seem mysterious at a first glance, it is quite natural. We will later provide an example of a more concrete condition which implies (SA5), and stick to the abstract setting for now.

The following lemma lists useful properties of $F$ in view of~(SA5).

\begin{lem}\label{lem:F_properties}
The function $F$ satisfies
\beqn
\int F\,\rd\bfP^{(2)} = 0
\eeqn
and
\beq\label{eq:Fcov}
\begin{split}
\int F\cdot F\circ(\Phi^{(2)})^i\,\rd\bfP^{(2)} = 2\int f(x)f(\varphi(i,\omega,x)) - f(x)f(\varphi(i,\omega,y))\,\rd\bfP^{(2)}(\omega,x,y).
\end{split}
\eeq
The latter has the upper bound
\beq\label{eq:Fcov_bound}
\left|\int F\cdot F\circ(\Phi^{(2)})^i\,\rd\bfP^{(2)}\right| \le 2\eta(i).
\eeq
\end{lem}
\begin{proof}
That~$F$ is centered is due to the symmetry property~\eqref{eq:P^2_symmetry} of~$\bfP^{(2)}$ in~(SA5). Since
\beqn
(F\cdot F\circ(\Phi^{(2)})^i)(\omega,x,y) = \{f(x)-f(y)\}\{f(\varphi(i,\omega,x)) - f(\varphi(i,\omega,y))\},
\eeqn
the same symmetry property also yields~\eqref{eq:Fcov}. The upper bound in~\eqref{eq:Fcov_bound} then follows from~\eqref{eq:SA5_lim2} and~\eqref{eq:SA5_lim3} in~(SA5) together with~(SA1).
\end{proof}

Recall that in Theorems~\ref{thm:sigma2n conv}, \ref{thm: sigma_n^2 to sigma^2} and~\ref{thm:main} we have
\beqn
\sigma^2 = \lim_{n\to\infty}\bE\sigma^2_n = \sum_{k=0}^\infty (2-\delta_{k0})\lim_{i\to\infty}    \bar{\bE} [\mu(f_i f_{i+k}) - \mu(f_i)\mu(f_{i+k})].
\eeqn
The next lemma connects this expression to the RDS notions when also (SA5) is assumed.

\begin{lem}
The limit variance $\sigma^2$ in Theorems~\ref{thm:sigma2n conv}, \ref{thm: sigma_n^2 to sigma^2} and~\ref{thm:main} satisfies
\begin{align}\label{eq:var_1}
\sigma^2 & = \sum_{k=0}^\infty (2-\delta_{k0}) \int f(x)f(\varphi(k,\omega,x)) - f(x)f(\varphi(k,\omega,y))\,\rd\bfP^{(2)}(\omega,x,y)
\\
& = \frac12 \sum_{k=0}^\infty(2-\delta_{k0})\int F\cdot F\circ(\Phi^{(2)})^k\,\rd\bfP^{(2)} 
\label{eq:var_2}
\\
& = \frac12 \lim_{n\to\infty} \frac1n \int Z_n^2\,\rd\bfP^{(2)} = \frac12 \lim_{n\to\infty} \Var_{\bfP^{(2)}}\!\left(\frac{Z_n}{\sqrt n}\right) .
\label{eq:var_3}
\end{align}
\end{lem}
\begin{proof}
The first line is just the expression of $\sigma^2$ rewritten using~\eqref{eq:SA5_lim2} and~\eqref{eq:SA5_lim3}. The second line then follows by~\eqref{eq:Fcov}.
The last line holds by~\eqref{eq:Zn_sum} together with~\eqref{eq:Fcov_bound} and~(SA1). 
\end{proof}

\begin{remark}\label{rem:Kubo1}
Note that the expression of~$\sigma^2$ in~\eqref{eq:var_2} is exactly {\bf one half} of the Green--Kubo formula in terms of the skew-product~$\Phi^{(2)}$, its invariant measure~$\bfP^{(2)}$, and the observable~$F$. This trick of ``doubling the dimension'' is not new. To our knowledge, however,~\eqref{eq:var_2} is a new observation at this level of generality. It answers a question raised in Section~7 of~\cite{AiminoNicolVaienti_2015} by Aimino, Nicol and Vaienti (who studied the special case where $\bP$, $\bfP$ and $\bfP^{(2)}$ are product measures, allowing for a non-random centering of $S_n$): The key that makes~\eqref{eq:var_2} an algebraic fact is the {\bf symmetry} property~\eqref{eq:P^2_symmetry} of the measure~$\bfP^{(2)}$.

It deserves a separate remark that even though $\sigma^2$ does not in general (see Remark~\ref{rem:Kubo2}) admit a classical Green-Kubo formula in terms of $\Phi$, $\bfP$, and $f$, ``doubling the dimension'' still yields~\eqref{eq:var_2}.
\end{remark}


\section{Positivity of $\sigma^2$}\label{sec:pos_var}
In this section we return to the question of positivity of the limit variance $\sigma^2$. We shall assume (SA1) and (SA3)--(SA5), the strong-mixing assumption (SA2) being unnecessary here. Again we assume nice parameters --- e.g.~$\psi>2$ --- for simplicity of the statements.

The foregoing discussion allows us to give some characterizations of the cases $\sigma^2 = 0$ and $\sigma^2>0$ on various levels of abstraction:
\begin{lem}\label{lem:zerovariance}
Suppose~$\eta(n) = Cn^{-\psi}$, $\psi>2$. 
\begin{enumerate}[(i)]
\item $\sigma^2 = 0$ is equivalent to each of the following conditions:
\smallskip  
\begin{enumerate}
\item 
$\sup_{n\ge 0}\int Z_n^2\,\rd\bfP^{(2)}<\infty$.
\item There exists $G\in L^2(\bfP^{(2)})$ such that $F = G-G\circ \Phi^{(2)}$.
\end{enumerate}

\medskip
\item $\sigma^2>0$ is equivalent to each of the following conditions: 
\smallskip
\begin{enumerate}
\item $\sup_{n\ge 0}\int Z_n^2\,\rd\bfP^{(2)}=\infty$.
\item There exist $c>0$ and $N>0$ such that $\int Z_n^2\,\rd\bfP^{(2)} \ge cn$ for all $n\ge N$.
\end{enumerate}

\medskip
\item If $\zeta>1$, then $\sigma^2>0$ is equivalent to each of the following conditions:
\smallskip
\begin{enumerate}
\item $\sup_{n\ge 1}n^{-\frac{1}{\psi}} \int Z_n^2\,\rd(\bP\otimes\mu\otimes\mu) = \infty$.
\item $\sup_{n\ge 1}n^{-\frac{1}{\psi}}\, \bE\Var_\mu(S_n) = \infty$.
\item There exist $c>0$ and $N>0$ such that $\int Z_n^2\,\rd(\bP\otimes\mu\otimes\mu) \ge cn$ for all $n\ge N$.
\item There exist $c>0$ and $N>0$ such that $\bE\Var_\mu(S_n) \ge cn$ for all $n\ge N$.
\end{enumerate}

\medskip
\item If $\bP$ is stationary, then $\sigma^2=0$ is equivalent to each of the following conditions:
\smallskip
\begin{enumerate}
\item $\sup_{n\ge 1}\int Z_n^2\,\rd(\bP\otimes\mu\otimes\mu) < \infty$.
\item $\sup_{n\ge 1} \bE\Var_\mu(S_n) < \infty$.
\end{enumerate}

\medskip
\item If $\bP$ is stationary, then $\sigma^2>0$ is equivalent to each of the following conditions:
\smallskip
\begin{enumerate}
\item $\sup_{n\ge 1} \int Z_n^2\,\rd(\bP\otimes\mu\otimes\mu) = \infty$.
\item $\sup_{n\ge 1} \bE\Var_\mu(S_n) = \infty$.
\item There exist $c>0$ and $N>0$ such that $\int Z_n^2\,\rd(\bP\otimes\mu\otimes\mu) \ge cn$ for all $n\ge N$.
\item There exist $c>0$ and $N>0$ such that $\bE\Var_\mu(S_n) \ge cn$ for all $n\ge N$.
\end{enumerate}
\end{enumerate}
\end{lem}

From the point of view of applications, parts~(iii)(b\&d), (iv)(b) and~(v)(b\&d) may be the most relevant ones as they involve the measures~$\bP$ and~$\mu$, and the process $(S_n)_{n\ge 1}$, which are immediately apparent from the definition of the system. Note that~(iii)(b) is the same condition as in Corollary~\ref{cor:nonzero_variance}(ii).

\begin{proof}[Proof of Lemma~\ref{lem:zerovariance}]
By~\eqref{eq:Fcov_bound} we can appeal to a well-known result due to Leonov~\cite{Leonov_1961}, which guarantees that
the limit
$
b = \lim_{n\to\infty}\int Z_n^2\,\rd\bfP^{(2)}
$
exists in $[0,\infty]$, and $b<\infty$ if and only if $\sup_{n\ge 0}\int Z_n^2\,\rd\bfP^{(2)}<\infty$. Moreover, the last condition is equivalent to the existence of $G\in L^2(\bfP^{(2)})$ such that $F = G-G\circ \Phi^{(2)}$. On the other hand, standard computations and the formula for~$\sigma^2$ in~\eqref{eq:var_2} yield
\beqn
\begin{split}
\int Z_n^2\,\rd\bfP^{(2)} & = 2\sigma^2 n - 2n\sum_{k=n}^\infty \int F\cdot F\circ(\Phi^{(2)})^k\,\rd\bfP^{(2)} - 2\sum_{k=1}^{n-1} k \int F\cdot F\circ(\Phi^{(2)})^k\,\rd\bfP^{(2)}
\\
& =  2\sigma^2 n + O\!\left(n \sum_{k=n}^\infty k^{-\psi} + \sum_{k=1}^{n-1} k^{1-\psi}\right)
=  2\sigma^2 n + O(1).
\end{split}
\eeqn
Here $\psi>2$ was used.
Thus, $\sigma^2>0$ is equivalent to linear growth of $\int Z_n^2\,\rd\bfP^{(2)}$ to infinity, while $\sigma^2 = 0$ is equivalent to $\sup_{n\ge 0}\int Z_n^2\,\rd\bfP^{(2)}<\infty$. Parts~(i) and~(ii) are proved.

As for part~(iii), \eqref{eq:varZ} and Theorem~\ref{thm: sigma_n^2 to sigma^2} with $\zeta>1$ yield 
\beqn
\int Z_n^2\,\rd(\bP\otimes\mu\otimes\mu) = 2\,\bE\Var_\mu(S_n) = 2(\sigma^2+O(n^{\frac{1}{\psi}-1})) n = 2\sigma^2 n + O(n^{\frac{1}{\psi}}).
\eeqn
If $\sigma^2 = 0$, the right side of the estimate is $O(n^{\frac{1}{\psi}})$, and each of the conditions (a)--(d) fails. If $\sigma^2>0$, the right side grows asymptotically linearly in~$n$, and (a)--(d) are all satisfied.

Finally, parts~(iv) and~(v) follow from~(i) and~(ii), respectively, because in the stationary case it holds that $\int Z_n^2\,\rd(\bP\otimes\mu\otimes\mu) = \int Z_n^2\,\rd\bfP^{(2)} + O(1)$; see Lemma~\ref{lem:varZ_stationary} below.
\end{proof}

We close the section with the lemma below, which was needed in the last part of the preceding proof.
\begin{lem}\label{lem:varZ_stationary}
Suppose $\bP$ is stationary and~$\eta(n) = Cn^{-\psi}$, $\psi>2$. Then 
\beqn
\sup_{n\ge 1}\left|\int Z_n^2\,\rd(\bP\otimes\mu\otimes\mu) - \int Z_n^2\,\rd\bfP^{(2)}\right| < \infty.
\eeqn
\end{lem}
\begin{proof}
We have
\beq
\int Z_n^2\,\rd(\bP\otimes\mu\otimes\mu)= \sum_{i=0}^{n-1}\sum_{k=0}^{n-i-1}(2-\delta_{k0})\int F\circ(\Phi^{(2)})^i\, F\circ(\Phi^{(2)})^{i+k}\,\rd(\bP\otimes\mu\otimes\mu)\label{eq:Z_n P mu mu}
\eeq
and
\beq
\begin{split}
\int Z_n^2\,\rd\bfP^{(2)} & = \sum_{i=0}^{n-1}\sum_{k=0}^{n-i-1}(2-\delta_{k0})\int F\circ(\Phi^{(2)})^i \,F\circ(\Phi^{(2)})^{i+k}\,\rd\bfP^{(2)} 
\\
 & =\sum_{i=0}^{n-1}\sum_{k=0}^{n-i-1}(2-\delta_{k0})\int F\cdot F\circ(\Phi^{(2)})^{k}\,\rd\bfP^{(2)}.\label{eq:Z_n P^2}
\end{split}
\eeq
Denote 
\beqn
\begin{split}
a_{ik} & = \int F\circ(\Phi^{(2)})^i\, F\circ(\Phi^{(2)})^{i+k}\,\rd(\bP\otimes\mu\otimes\mu)
 = 2\bE[\mu(f_{i}f_{i+k})]- 2\bE[\mu(f_{i})\mu(f_{i+k})]
\end{split}
\eeqn
and 
\beqn
 b_k =\int F\cdot F\circ(\Phi^{(2)})^{k}\,\rd\bfP^{(2)}.
\eeqn
Note that $|a_{ik}|\le 2\eta(k)$ by (SA1) and $|b_k|\le 2\eta(k)$ by \eqref{eq:Fcov_bound}. Thus $|a_{ik}-b_k|\le Ck^{-\psi}$ for all $i$ and $k$. By stationarity ($\bP = \bar\bP$) and (SA5),
 \beqn
 \begin{split}
 \lim_{i\to \infty} a_{ik}
= b_k.
 \end{split}
\eeqn
Again by stationarity, \eqref{eq:bar E conv} implies that $|a_{ik}-b_k|\le C\eta(i)=Ci^{-\psi}.$ Thus
\beq
|a_{ik}-b_k|\le C\max\{i,k\}^{-\psi}.\label{eq:a_ik-b_k}
\eeq
Collecting \eqref{eq:Z_n P mu mu}, \eqref{eq:Z_n P^2} and \eqref{eq:a_ik-b_k} we get the estimate
\beqn
 \begin{split}
 &\left|\int Z_n^2\,\rd(\bP\otimes\mu\otimes\mu) - \int Z_n^2\,\rd\bfP^{(2)}\right| 
 \\
  \le\ &\sum_{i=0}^{n-1}\sum_{k=0}^{n-i-1}(2-\delta_{k0})|a_{ik}-b_k|
  \le C\sum_{i=0}^{\infty}\sum_{k=0}^{\infty} \max\{i,k\}^{-\psi}
  = C\sum_{k=0}^{\infty}(2k+1)k^{-\psi} < \infty 
 \end{split}
\eeqn
for all $n$. The proof is complete.
\end{proof}


\section{Effect of the fiberwise centering of $W_n$}\label{sec:centerings}
In this section we discuss Remark~\ref{rem:variances} concerning the variance of~$W_n$, as opposed to the fiberwise-centered $\bar W_n = W_n - \mu(W_n)$. Note that
\beq\label{eq:varW}
\Var_{\bP\otimes\mu} W_n = \bE\mu(W_n^2) - (\bE\mu(W_n))^2
\eeq
and
\beq\label{eq:varcenter}
\Var_\bP \mu(W_n) = \bE[\mu(W_n)^2] - (\bE\mu(W_n))^2.
\eeq
The difference of \eqref{eq:varW} and $\eqref{eq:varcenter}$ equals
\beqn
\bE\sigma_n^2 = \Var_{\bP\otimes\mu} \bar W_n =  \bE\mu(W_n^2) -  \bE[\mu(W_n)^2].
\eeqn
Under the assumptions of our paper
\beqn
\bE\sigma_n^2 = \sigma^2 + o(1).
\eeqn
Therefore $\Var_{\bP\otimes\mu} W_n$ and $\Var_\bP \mu(W_n)$ either converge or diverge simultaneously. We now derive their asymptotic expressions in terms of series, restricting to the case where the law $\bP$ of the selection process is {\bf stationary}.

\begin{lem}\label{lem:other_variances}
Let $\bP$ be stationary. Let (SA1)--(SA3) hold with $\eta(n) = Cn^{-\psi}$, $\psi>1$, and $\alpha(n) = n^{-\gamma}$, $\gamma>1$. Then
\beqn
\Var_{\bP\otimes\mu} W_n =  \sum_{k=0}^\infty (2-\delta_{k0})\lim_{i\to\infty}\{\bE[\mu(f_{i}f_{i+k})]- \bE \mu(f_{i})\bE \mu(f_{i+k})\} + O(n^{\frac1{\min\{\gamma,\psi\}}-1})
\eeqn
and
\beqn
\Var_\bP\mu(W_n)  =  \sum_{k=0}^\infty (2-\delta_{k0})\lim_{i\to\infty}\{\bE[\mu(f_{i})\mu(f_{i+k})]- \bE \mu(f_{i})\bE \mu(f_{i+k})\} +O(n^{\frac1{\min\{\gamma,\psi\}}-1}).
\eeqn
Here the limits exist and the series converge absolutely. If also (SA5) holds, then
\beqn
\begin{split}
& \lim_{n\to\infty}\Var_{\bP\otimes\mu} W_n 
\\
=\ & \sum_{k=0}^\infty (2-\delta_{k0})\left\{ \int f(x)\, f(\varphi(k,\omega,x))\,\rd\bfP(\omega,x)- \left(\int f(x)\, \rd\bfP(\omega,x)\right)^2\right\}
\end{split}
\eeqn
and
\beqn
\begin{split}
& \lim_{n\to\infty}\Var_\bP\mu(W_n)
\\
=\ & \sum_{k=0}^\infty (2-\delta_{k0})\left\{ \int f(x)\, f(\varphi(k,\omega,y))\,\rd\bfP^{(2)}(\omega,x,y)- \left(\int f(x)\, \rd\bfP(\omega,x)\right)^2\right\}
\end{split}
\eeqn
using the RDS notations.
\end{lem}

\begin{remark}\label{rem:Kubo2}
Note that in the latter case
\beqn
\lim_{n\to\infty}\Var_{\bP\otimes\mu} W_n = \sum_{k=0}^\infty (2-\delta_{k0}) \Cov_{\bfP}(f,f\circ\Phi^k).
\eeqn
This is the classical Green--Kubo formula in terms of the skew-product~$\Phi$, its invariant measure~$\bfP$, and the observable~$f$. Let us stress that it is not the expression of~$\sigma^2$, save for exactly the special case $\lim_{n\to\infty}\Var_\bP\mu(W_n) = 0$. The latter special case is the very same in which Abdelkader and Aimino \cite{AbdelkaderAimino_2016} establish a quenched central limit theorem with non-random centering, assuming i.i.d.\@ randomness ($\bP = \bP_0^\bN$) in particular; see also Remark~\ref{rem:Kubo1}.
\end{remark}

\begin{proof}[Proof of Lemma~\ref{lem:other_variances}]
We prove the statements concerning~$\Var_\bP\mu(W_n)$ first.
We have
\beqn
\begin{split}
\bE[\mu(W_n)^2] - (\bE\mu(W_n))^2 
&=  \frac1n\sum_{i=0}^{n-1}\sum_{k=0}^{n-i-1} a_{ik}
\end{split}
\eeqn
where
\beqn
a_{ik}=(2-\delta_{k0})\{\bE[\mu(f_{i})\mu(f_{i+k})]- \bE \mu(f_{i})\bE \mu(f_{i+k})\}.
\eeqn
We will apply Lemma \ref{lem:mean} to show convergence as $n\to\infty$. To that end, we need control of $a_{ik}$ in the limits $i\to\infty$ and $k\to\infty$. We begin with the first limit.
 
By~\eqref{eq:memloss_cocycle} below (SA3), we have a uniform bound
\beq\label{eq:mufi}
|\mu(f\circ\varphi(i,\omega))-\mu(f\circ\varphi(r,\tau^{i-r}\omega))| \le C\eta(r)
\eeq
whenever $r\le i$. Since $\bP$ is stationary, this yields
\beqn
|\bE\mu(f_{i})-\bE\mu(f_{r})| \le C\eta(r).
\eeqn
Thus, $(\bE\mu(f_{i}))_{i= 0}^\infty$ is Cauchy, so its limit exists and
\beqn
\left|\lim_{i\to\infty}\bE\mu(f_{i})\bE\mu(f_{i+k})-\bE\mu(f_{r})\bE\mu(f_{r+k})\right| \le C\eta(r).
\eeqn
Since $\bar \bP = \bP$ by stationarity, \eqref{eq:bar E conv} gives
\beqn
\left|\lim_{i\to\infty}\bE[\mu(f_{i})\mu(f_{i+k})]-\bE[\mu(f_{r})\mu(f_{r+k})]\right| \le C\eta(r).
\eeqn
Thus, the limit
\beqn
b_k= (2-\delta_{k0})\lim_{i\to\infty}\{\bE[\mu(f_{i})\mu(f_{i+k})]- \bE \mu(f_{i})\bE \mu(f_{i+k})\}
\eeqn
exists and
\beqn
a_{ik} = b_k + O(\eta(i))
\eeqn
as $i\to\infty$. Since $\eta$ is summable,
\beqn
r_k(n) =  \frac1n\sum_{i = 0}^{n-1} a_{ik} - b_k = \frac1n\sum_{i=0}^{n-1}O(\eta(i)) = O(n^{-1})
\eeqn
as $n\to\infty$. Both of the preceding bounds are uniform in $k$.

In order to bound $a_{ik}$ as $k\to\infty$, first note that~\eqref{eq:mufi} allows to estimate
\beqn
|\mu(f_{i+k}) - v| \le C\eta(k-r)
\eeqn
for $r\le k$,
where the function $v(\omega) = \mu(f\circ\varphi(k-r,\tau^{i+r}\omega))$ is $\cF_{i+r+1}^\infty$-measurable and bounded; see~Section~\ref{sec:selection_process} for terminology.
Thus
\beqn
a_{ik} = (2-\delta_{k0})\{\bE[\mu(f_{i})v]- \bE \mu(f_{i})\bE v\} + O(\eta(k-r)) = O(\alpha(r)) + O(\eta(k-r)),
\eeqn
the last estimate being true by strong mixing. Picking $r\asymp k/2$ yields
\beqn
a_{ik} = O(k^{-\min\{\gamma,\psi\}})
\eeqn
uniformly in $i$. Since $\gamma>1$ and $\psi>1$, this bound is summable, so Lemma~\ref{lem:mean} can now be applied; recall~\eqref{eq:standing}. The bound in~\eqref{eq:mean-rate} becomes
\beqn
\left|\bE[\mu(W_n)^2] - (\bE\mu(W_n))^2  - \sum_{k=0}^\infty b_k\right| \le C(K^{1-\min\{\gamma,\psi\}} + Kn^{-1}).
\eeqn
Now, choosing $K\asymp n^{1/\min\{\gamma,\psi\}}$ yields the upper bound~$Cn^{1/\min\{\gamma,\psi\}-1}$ claimed.

The expressions of the limits~$b_k$ in term of the RDS notations is obtained with the help of \eqref{eq:SA5_lim1}--\eqref{eq:SA5_lim3}, recalling again $\bP = \bar\bP$ due to stationarity.

Finally, the claims regarding $\Var_{\bP\otimes\mu} W_n = \Var_{\bP\otimes\mu} \bar W_n + \Var_\bP\mu(W_n)$ follow since we already have control of both terms on the right side: in the stationary case at hand, Theorem~\ref{thm: sigma_n^2 to sigma^2} applies with any $\zeta>1$, yielding $\Var_{\bP\otimes\mu} \bar W_n = \sigma^2 + O(n^{\frac1\psi -1})$.
\end{proof}


\section{(SA5'): a less abstract substitute for (SA5)}\label{sec:SA5'}
Standing Assumption (SA5) is abstract in that it involves the invariant measure~$\bfP^{(2)}$ of the RDS~$\varphi^{(2)}$, and a number of properties of the measure, which are not obvious from the setup of the system at the beginning of the paper. For that reason we give in this section, as an example, another assumption which (i)~is more concrete in that it involves only the initial measure~$\mu$ and the basic cocycle~$\varphi$, and (ii)~is stronger than (SA5).

\medskip
\noindent{\bf Standing Assumption (SA5').} Throughout this section we assume following: the measures $\varphi(n,\omega)_*\mu$ have uniformly square integrable densities with respect to $\mu$, i.e., there exists $K>0$ such that
\beq\label{eq:SA5'.1}
\left\|\frac{\rd\varphi(n,\omega)_*\mu}{\rd\mu}\right\|_{L^2(\mu)} \le K
\eeq
for all $n$ and $\omega$.
Moreover, for every bounded measurable $g:X\to\bR$ and $\ve>0$ there exists $N\ge 0$ such that the memory-loss property
\beq\label{eq:SA5'.2}
\left|\int g(\varphi(n+m,\omega)x)\,\rd\mu(x) - \int g(\varphi(n,\tau^m \omega)x)\,\rd\mu(x)\right| < \ve
\eeq
hold for $n\ge N$, $m\ge 0$ and all~$\omega$.\hfill$\blacksquare$

\medskip

The rest of the section is devoted to investigating some consequences of (SA5').

Note that~\eqref{eq:SA5'.2} asks that the integrals of $x\mapsto g((n,\tau^m\omega)x)$ with respect to the two measures~$\varphi(m,\omega)_*\mu$ and~$\mu$ are essentially the same for large~$n$, uniformly in~$m$ and~$\omega$. The role of~\eqref{eq:SA5'.1} is to allow for uniform approximations of the compositions $h\circ(\Phi^{(2)})^n$, $n\ge 0$, by compositions $\hat h\circ(\Phi^{(2)})^n$, where $h$ is measurable and $\hat h$ is ``simple'': observe that $(h-\hat h)\circ(\Phi^{(2)})^n$ is not guaranteed to be uniformly (in $n$) small in $L^1(\bar\bP\otimes\mu\otimes\mu)$, even if $h-\hat h$ is small, without some assumption. To that end, let us already prove a little lemma:

\begin{lem}\label{lem:uniform_approximation}
Let $h:\Omega\times X\times X\to \bR$ belong to $L^2(\bar\bP\otimes\mu\otimes\mu)$. Then
\beqn
\|h\circ(\Phi^{(2)})^n\|_{L^1(\bar\bP\otimes\mu\otimes\mu)} \le K^2\|h\|_{L^2(\bar\bP\otimes\mu\otimes\mu)}
\eeqn
holds for all $n\ge 0$ with $K$ as in~\eqref{eq:SA5'.1}.
\end{lem}
\begin{proof}
Write $\lambda = \bar\bP\otimes\mu\otimes\mu$ for brevity. Observe that
\beqn
\begin{split}
& \left(\int|h|\circ(\Phi^{(2)})^n\,\rd\lambda\right)^2
\\
=\ & \left(\int |h|(\tau^n\omega,x,y)\,\frac{\rd\varphi(n,\omega)_*\mu}{\rd\mu}(x)\,\frac{\rd\varphi(n,\omega)_*\mu}{\rd\mu}(y)\,\rd\lambda(\omega,x,y)\right)^2
\\
\le\ & \int |h|^2(\tau^n\omega,x,y)\,\rd\lambda(\omega,x,y) \int \left|\frac{\rd\varphi(n,\omega)_*\mu}{\rd\mu}(x)\,\frac{\rd\varphi(n,\omega)_*\mu}{\rd\mu}(y)\right|^2\,\rd\lambda(\omega,x,y)
\end{split}
\eeqn
by H\"older's inequality. Here
\beqn
\int|h|^2(\tau^n\omega,x,y)\,\rd\lambda(\omega,x,y) = \int|h|^2(\omega,x,y)\,\rd\lambda(\omega,x,y)
\eeqn
since $\bar\bP$ is stationary. On the other hand,
\beqn
\begin{split}
& \int \left|\frac{\rd\varphi(n,\omega)_*\mu}{\rd\mu}(x)\,\frac{\rd\varphi(n,\omega)_*\mu}{\rd\mu}(y)\right|^2\,\rd\lambda(\omega,x,y)
\\
=\ & \int \left[\int \left|\frac{\rd\varphi(n,\omega)_*\mu}{\rd\mu}(x)\right|^2\,\rd\mu(x)\,\int \left|\frac{\rd\varphi(n,\omega)_*\mu}{\rd\mu}(y)\right|^2 \,\rd\mu(y)\right]\rd\bar\bP(\omega)
\\
\le\ & K^4
\end{split}
\eeqn
by~\eqref{eq:SA5'.1}. Combining the estimates and taking square roots yields the result.
\end{proof}

\subsection{Standing Assumption (SA5') implies (SA5)} 

\begin{lem}\label{lem:skew_strongconvergence}
There exists an invariant measure~$\bfP^{(2)}$ for the RDS $\varphi^{(2)}$ such that
\beqn
\lim_{n\to\infty }\int h\circ(\Phi^{(2)})^n\,\rd(\bar\bP\otimes\mu\otimes\mu) = \int h\,\rd\bfP^{(2)}
\eeqn
for all bounded measurable $h:\Omega\times X\times X\to\bR$. Moreover, \eqref{eq:P^2_symmetry} holds, and~$\bfP$ in  \eqref{eq:marginals_P} is an invariant measure for the RDS $\varphi$ such that
\beqn
\lim_{n\to\infty }\int \tilde h\circ\Phi^n\,\rd(\bar\bP\otimes\mu) = \int \tilde h\,\rd\bfP
\eeqn
for all bounded measurable $\tilde h:\Omega\times X\to\bR$. 
\end{lem}

\begin{proof}
Let $u:\Omega\to\bR$ and $g^1,g^2:X\to\bR$ be bounded measurable. Let $\ve>0$. Then there exists $N\ge 0$ such that
\beqn
\begin{split}
& \iiint (u\otimes g^1\otimes g^2)\circ (\Phi^{(2)})^{n+m}(\omega,x,y)\, \rd\mu(x)\, \rd\mu(y)\,\rd\bar\bP(\omega) 
\\
=\ & \int u(\tau^{n+m}\omega) \int g^1(\varphi(n+m,\omega)x)\, \rd\mu(x) \int g^2(\varphi(n+m,\omega)y)\, \rd\mu(y)\,\rd\bar\bP(\omega)
\\
=\ & \int u(\tau^{n+m}\omega) \int g^1(\varphi(n,\tau^m\omega)x)\, \rd\mu(x) \int g^2(\varphi(n,\tau^m\omega)y)\, \rd\mu(y)\,\rd\bar\bP(\omega) + O(\ve)
\\
=\ & \int u(\tau^{n}\omega) \int g^1(\varphi(n,\omega)x)\, \rd\mu(x) \int g^2(\varphi(n,\omega)y)\, \rd\mu(y)\,\rd\bar\bP(\omega) + O(\ve)
\\
=\ & \iiint (u\otimes g^1\otimes g^2)\circ (\Phi^{(2)})^n(\omega,x,y)\, \rd\mu(x)\, \rd\mu(y)\,\rd\bar\bP(\omega)  + O(\ve)
\end{split}
\eeqn
for all $n\ge N$ and $m\ge 0$. Here the third line uses~\eqref{eq:SA5'.2} and the fourth line uses stationarity. Thus, we see that the sequence $(\iiint (u\otimes g^1\otimes g^2)\circ (\Phi^{(2)})^{n}(\omega,x,y)\, \rd\mu(x)\, \rd\mu(y)\,\rd\bar\bP(\omega) )_n$ is Cauchy and therefore convergent. We will show using the monotone class theorem that the convergence property extends to an arbitrary bounded measurable function in place of $u\otimes g^1\otimes g^2$.

Let $\cH$ denote the set of all measurable functions $h:\Omega\times X\times X\to\bR$ such that $\lim_{n\to\infty }\int h\circ(\Phi^{(2)})^n\,\rd(\bar\bP\otimes\mu\otimes\mu)$ exists. Let $\cA$ denote the set of all measurable cubes in $\Omega\times X\times X$. Clearly~$\cA$ is nonempty and closed under finite intersections, and it contains the product space $\Omega\times X\times X$. Clearly~$\cH$ is closed under linear combinations. Furthermore, the argument above shows~$1_A\in\cH$ for all~$A\in\cA$. Suppose now that~$h_k\in\cH$ are nonnegative functions increasing to a bounded function~$h$. Showing~$h\in\cH$ proves that~$\cH$ contains all bounded functions that are measurable with respect to the sigma-algebra~$\sigma(\cA) = \cF\otimes\cB\otimes\cB$. We will show $h\in\cH$ next.

Let $\ve>0$ be fixed. Since $0\le h_k\uparrow h$ where $h$ is bounded, by the bounded convergence theorem there exists $k_0 =  k_0(\ve)$ such that
$
\|h-h_{k_0}\|_{L^2(\bar\bP\otimes\mu\otimes\mu)} < \ve.
$ 
Thus, by Lemma~\ref{lem:uniform_approximation},
\beqn
\|(h-h_{k_0})\circ(\Phi^{(2)})^n\|_{L^1(\bar\bP\otimes\mu\otimes\mu)} < K^2\ve
\eeqn
for all $n\ge 1$. Since $h_{k_0}\in\cH$, there exists $n_0 = n_0(\ve)$ such that
\beqn
\left|\int h_{k_0}\circ(\Phi^{(2)})^n\,\rd(\bar\bP\otimes\mu\otimes\mu) - \lim_{m\to\infty}\int h_{k_0} \circ(\Phi^{(2)})^m\,\rd(\bar\bP\otimes\mu\otimes\mu)\right| < \ve
\eeqn
for all $n\ge n_0$. A combination of the estimates yields
\beqn
\left|\int h\circ(\Phi^{(2)})^n\,\rd(\bar\bP\otimes\mu\otimes\mu) - \lim_{m\to\infty}\int h_{k_0} \circ(\Phi^{(2)})^m\,\rd(\bar\bP\otimes\mu\otimes\mu)\right| < K^2\ve + \ve
\eeqn
for all $n\ge n_0$.
Hence $h\in\cH$. Therefore, by the monotone class theorem $\cH$ contains all bounded measurable functions.

By the Vitali--Hahn--Saks theorem there exists a probability measure~$\bfP^{(2)}$ satisfying
\beqn
\lim_{n\to\infty }\int h\circ(\Phi^{(2)})^n\,\rd(\bar\bP\otimes\mu\otimes\mu) = \int h\,\rd\bfP^{(2)}
\eeqn
for all bounded measurable $h:\Omega\times X\times X\to\bR$. The symmetry property~\eqref{eq:P^2_symmetry} of~$\bfP^{(2)}$ is an immediate consequence.
By the same token~$\bfP^{(2)}$ is invariant for~$\Phi^{(2)}$:
\beqn
\int h\circ \Phi^{(2)}\,\rd\bfP^{(2)} = \lim_{n\to\infty }\int h\circ\Phi^{(2)}\circ(\Phi^{(2)})^n\,\rd(\bar\bP\otimes\mu\otimes\mu) = \int h\,\rd\bfP^{(2)}.
\eeqn
Furthermore, taking $h$ of the form $h(\omega,x,y) = u(\omega)$,
\beqn
\int h\,\rd\bfP^{(2)} = \lim_{n\to\infty} \int u(\tau^n\omega)\, \rd\bar\bP(\omega) = \int u\, \rd\bar\bP
\eeqn
shows $(\Pi_1)_*\bfP^{(2)} = \bar\bP$. Thus, $\bfP^{(2)}$ is an invariant measure for the RDS~$\varphi^{(2)}$. 

Suppose that either $h(\omega,x,y) = \tilde h(\omega,x)$ or $h(\omega,x,y) = \tilde h(\omega,y)$ holds identically. Then
\beqn
\begin{split}
\int\tilde h\,\rd\bfP = \int h\,\rd\bfP^{(2)} & = \lim_{n\to\infty }\int h\circ(\Phi^{(2)})^n\,\rd(\bar\bP\otimes\mu\otimes\mu) 
 = \lim_{n\to\infty }\int \tilde h\circ\Phi^n\,\rd(\bar\bP\otimes\mu) .
\end{split}
\eeqn
This yields the claims concerning~$\bfP$.
\end{proof}

We are in position to prove the promised fact:
\begin{lem}\label{lem:skewlimits}
Standing Assumption (SA5') implies (SA5).
\end{lem}
\begin{proof}
By Lemma~\ref{lem:skew_strongconvergence} it remains to verify~\eqref{eq:SA5_lim1}--\eqref{eq:SA5_lim3}. Using Lemma~\ref{lem:skew_strongconvergence},
\beqn
\begin{split}
\lim_{i\to\infty}\bar\bE[\mu(f_i)] 
& = \lim_{i\to\infty} \int f\circ\Pi_{2}\circ\Phi^{i}\,\rd(\bar\bP\otimes\mu) 
 = \int f\circ\Pi_{2}\,\rd\bfP
\end{split}
\eeqn
and
\beqn
\begin{split}
\lim_{i\to\infty}\bar\bE[\mu(f_i f_{i+k})] 
& = \lim_{i\to\infty} \int (f\circ\Pi_{2}\ f\circ\Pi_{2}\circ\Phi^k)\circ\Phi^{i}\,\rd(\bar\bP\otimes\mu) 
= \int f\circ\Pi_{2}\ f\circ\Pi_{2}\circ\Phi^k\,\rd\bfP.
\end{split}
\eeqn
Likewise
\beqn
\begin{split}
\lim_{i\to\infty}\bar\bE[\mu(f_i) \mu(f_{i+k})]
&=  \lim_{i\to\infty} \int (f\circ\Pi_{2}\ f\circ\Pi_{3}\circ (\Phi^{(2)})^{k})\circ (\Phi^{(2)})^{i}\,\rd(\bar\bP\otimes \mu \otimes\mu)
\\
&=   \int f\circ\Pi_{2}\ f\circ\Pi_{3}\circ (\Phi^{(2)})^{k}\,\rd\bfP^{(2)}.
\end{split}
\eeqn
The proof is complete.
\end{proof}


\subsection{Disintegration of the invariant measure~$\bfP^{(2)}$}

In this subsection we shed some light on the invariant measure~$\bfP^{(2)}$ of the RDS~$\varphi^{(2)}$ with the aid of disintegrations. The mathematical constructions here are well known, and we include this part for completeness. The results call for nice structure of the measurable spaces: we assume that both~$(X,\cB)$ and $(\Omega_0,\cE)$ are {\bf standard measurable spaces}.

We begin by stating a basic fact:

\begin{lem}\label{lem:disintegration}
There exists a family of set functions $\nu^{(2)}_\omega:\cB\to[0,1]$,  $\omega\in\Omega$, such that
\begin{enumerate}[(i)]
\item the map $\omega\mapsto\nu^{(2)}_\omega(B)$ is measurable for all $B\in\cB\otimes\cB$;
\item $\nu^{(2)}_\omega$ is a probability measure for $\bar\bP$-a.e.\@ $\omega\in\Omega$;
\item for all $h\in L^1(\bfP^{(2)})$,
\beqn
\int h\,\rd\bfP^{(2)} = \int_\Omega\int_{X\times X} h(\omega,x,y)\,\rd\nu^{(2)}_{\omega}(x,y)\,\rd\bar\bP(\omega).
\eeqn
\end{enumerate}
The disintegration is essentially unique: if $\tilde\nu^{(2)}_\omega$, $\omega\in\Omega$, is another family of such set functions, then $\nu^{(2)}_\omega = \tilde\nu^{(2)}_\omega$ for $\bar\bP$-a.e.\@ $\omega\in\Omega$.
\end{lem}

\begin{proof}
Since the product space $(X\times X, \cB\otimes\cB)$ is also a standard measurable space and~$(\Pi_1)_*\bfP = \bar\bP$, classical results yield the lemma; see, e.g., Proposition~1.4.3 of Arnold~\cite{Arnold}. 
\end{proof}

It is helpful to think of $\nu^{(2)}_\omega$ as the conditional measure $\bfP^{(2)}(\slot|\,\omega)$. In the following we will characterize the conditional measures~$\nu^{(2)}_{\omega}$.

Next, we extend $\bar\bP$ to a stationary measure on the space of two-sided sequences. To that end define
$\Omega^- = \Omega_0^{\{\dots,-2,-1,0\}}$ and $\Omega^+ = \Omega_0^{\{1,2,3,\dots\}} = \Omega$. The sigma-algebras $\cF^-$ and $\cF^+=\cF$ denote the corresponding products of $\cE$. Write also
\beqn
\bar\Omega = \Omega^-\times\Omega^+ = \Omega_0^\bZ
\quad\text{and}\quad
\bar\cF = \cF^{-}\otimes\cF^{+} = \cE^\bZ.
\eeqn
Let $\bar\tau:\bar\Omega\to\bar\Omega$ denote the two-sided shift: $(\bar\tau^k\bar\omega)_i = \bar\omega_{i+k}$ for all $i,k\in\bZ$. Finally, let $\Pi^\pm:\bar\Omega\to\Omega^\pm$ denote the canonical projections: $\Pi^-(\bar\omega) = \omega^-$ and $\Pi^+(\bar\omega) = \omega^+$ for all $\bar\omega = (\omega^-,\omega^+)\in\Omega^-\times\Omega^+$. 

We are ready to state another basic fact:

\begin{lem}\label{lem:natural_extension}
(1) There exists a unique probability measure $\bar\bQ$ on $(\bar\Omega,\bar\cF)$ which is invariant for $\bar\tau$ and satisfies $(\Pi^+)_*\bar\bQ = \bar\bP$.

\noindent (2) There exists an essentially unique family of set functions $q_\omega:\cF^-\to[0,1]$,  $\omega\in\Omega$, such that
\begin{enumerate}[(i)]
\item the map $\omega\mapsto q_\omega(E)$ is measurable for all $E\in\cF^-$;
\item $q_\omega$ is a probability measure for $\bar\bP$-a.e.\@ $\omega\in\Omega$;
\item for all $h\in L^1(\bar\bQ)$,
\beqn
\int_{\bar\Omega} h(\bar\omega)\,\rd\bar\bQ(\bar\omega) = \int_\Omega\int_{\Omega^-} h(\omega^-,\omega)\,\rd q_{\omega}(\omega^-)\,\rd\bar\bP(\omega).
\eeqn
\end{enumerate}
\end{lem}

\begin{proof}
(1) Since $(\Omega_0,\cE)$ is a standard measurable space, the shift-invariant measure~$\bar\bQ$ having~$\bar\bP$ as its marginal is uniquely constructed with the aid of Kolmogorov's extension theorem by requiring that the finite dimensional distributions are translation invariant and coincide with those of $\bar\bP$.
See, e.g., Appendix~A.3 of Arnold~\cite{Arnold} for details.

(2) Since~$(\Omega^-,\cF^-)$ is a standard probability space and $(\Pi^+)_*\bar\bQ = \bar\bP$, the result is classical as in Lemma~\ref{lem:disintegration}.
\end{proof}
The resulting dynamical system $(\bar\Omega,\bar\cF,\bar\bQ,\bar\tau)$ is the natural extension of $(\Omega,\cF,\bar\bP,\tau)$ with homomorphism~$\Pi^+$. The intuition behind the measures in Lemma~\ref{lem:natural_extension} is the following: Think of $\omega = (\omega_1,\omega_2,\dots)$ as a stochastic process with law~$\bar\bP$. Due to stationarity, it is possible to glue a history $\omega^- = (\dots,\omega_{-1},\omega_0)$ to~$\omega$ in a consistent and unique way such that the law~$\bar\bQ$ of $\bar\omega = (\omega^-,\omega) = (\dots,\omega_{-1},\omega_0,\omega_1,\omega_2,\dots)$ is stationary and the marginal law corresponding to the future part~$\omega$ is~$\bar\bP$. The measure $q_\omega$ can be thought of as the conditional law $\bar\bQ(\slot|\,\omega)$, the distribution of the past~$\omega^{-}$ given the future~$\omega$. 

For the following it will be convenient to introduce the notations
\beqn
\varphi(\omega^-_{-n+1},\dots,\omega^-_0) = T_{\omega^-_0}\circ\dots\circ T_{\omega^-_{-n+1}}
\eeqn
and
\beqn
\varphi^{(2)}(\omega^-_{-n+1},\dots,\omega^-_0)(x,y) = (\varphi(\omega^-_{-n+1},\dots,\omega^-_0)x,\,\varphi(\omega^-_{-n+1},\dots,\omega^-_0)y)
\eeqn
for any finite sequence $(\omega^-_{-n+1},\dots,\omega^-_0)\subset\Omega_0$.

Now, for all bounded measurable functions $h(\omega,x,y) = u(\omega)g(x,y)$ we have
\beq\label{eq:fiber1}
\int h\,\rd\bfP^{(2)} = \int_\Omega u(\omega)\int_{X\times X} g\,\rd\nu^{(2)}_{\omega}\,\rd\bar\bP(\omega).
\eeq
On the other hand,
Lemma~\ref{lem:skew_strongconvergence} yields
\beqn
\begin{split}
\int h\,\rd\bfP^{(2)}
& = \lim_{n\to\infty} \int_\Omega u(\tau^n\omega) \int_{X\times X} g\circ\varphi^{(2)}(n,\omega)\,\rd(\mu\otimes\mu)\,\rd\bar\bP(\omega)
\\
& = \lim_{n\to\infty} \int_{\bar\Omega} u(\Pi^+(\bar\tau^n\bar\omega)) \int_{X\times X} g\circ\varphi^{(2)}(n,\Pi^+(\bar\omega))\,\rd(\mu\otimes\mu)\,\rd\bar\bQ(\bar\omega)
\\
& = \lim_{n\to\infty} \int_{\bar\Omega} u(\Pi^+(\bar\omega)) \int_{X\times X} g\circ\varphi^{(2)}(n,\Pi^+(\bar\tau^{-n}\bar\omega))\,\rd(\mu\otimes\mu)\,\rd\bar\bQ(\bar\omega).
\end{split}
\eeqn
In order to disintegrate $\bar\bQ$, let us write $\bar\omega = (\omega^-,\omega)$ in the obvious manner, noting that $u(\Pi^+(\bar\omega)) = u(\omega)$ and $\varphi^{(2)}(n,\Pi^+(\bar\tau^{-n}\bar\omega)) = \varphi^{(2)}(\omega^-_{-n+1},\dots,\omega^-_0)$. Thus, Lemma~\ref{lem:natural_extension} yields
\beq\label{eq:fiber2}
\begin{split}
\int h\,\rd\bfP^{(2)}
& = \lim_{n\to\infty} \int_{\Omega} u(\omega) \int_{\Omega^{-}} \int_{X\times X} g\circ\varphi^{(2)}(\omega^-_{-n+1},\dots,\omega^-_0)\,\rd(\mu\otimes\mu)\,\rd q_\omega(\omega^-)\,\rd\bar\bP(\omega).
\end{split}
\eeq
The following observation is now key:
\begin{lem}\label{lem:limit_from_past}
Given $\omega^-\in\Omega^-$, there exists a probability measure~$\mu_{\omega^{-}}$ on~$(X,\cB)$ such that
\beq\label{eq:limit_from_past}
\lim_{n\to\infty} \int_{X\times X} g\circ\varphi^{(2)}(\omega^-_{-n+1},\dots,\omega^-_0)\,\rd(\mu\otimes\mu) = \int_{X\times X} g\,\rd(\mu_{\omega^-}\otimes\mu_{\omega^-})
\eeq
for all bounded measurable~$g:X\times X\to\bR$. 
\end{lem}

Note that~$\mu_{\omega^-}$ has the interpretation of being the pushforward of~$\mu$ from the infinitely distant past along the history $\omega^- = (\dots,\omega^-_1,\omega^-_0)$.

\begin{proof}[Proof of Lemma~\ref{lem:limit_from_past}]
Consider first a bounded measurable $g^1:X\to\bR$. Let $\ve>0$. By~\eqref{eq:SA5'.2} of~(SA5') there exists~$N\ge 0$ such that 
\beq\label{eq: g_epsilon 1}
\left| \int g^1\circ \varphi(\omega^-_{-n-m+1},\dots,\omega^-_0)\,\rd \mu - \int g^1\circ \varphi(\omega^-_{-n+1},\dots,\omega^-_0)\,\rd \mu \right|< \ve
\eeq
for all $n\ge N$ and $m\ge 0$.
We see that $(\int g^1\circ \varphi(\omega^-_{-n+1},\dots,\omega^-_0)\,\rd \mu)_{n=1}^{\infty}$ is a Cauchy sequence, and thus convergences. Since $g^1$ was arbitrary, the Vitali--Hahn--Saks theorem yields the existence of a measure $\mu_{\omega^-}$ such that
\beqn
\lim_{n\to \infty}\int g^1\circ \varphi(\omega^-_{-n+1},\dots,\omega^-_0)\,\rd \mu = \int g^1 \,\rd\mu_{\omega^{-}}.
\eeqn
This yields~\eqref{eq:limit_from_past}
for all $g(x,y) = g^1(x)g^2(y)$ with both $g^1,g^2:X\to\bR$ bounded and measurable. Similarly to the proof of Lemma~\ref{lem:skew_strongconvergence}, a straightforward application of the monotone class theorem extends~\eqref{eq:limit_from_past} to all bounded measurable $g:X\times X\to\bR$.
\end{proof}

We finally arrive at the characterization of the conditional measure $\nu^{(2)}_\omega$ as the expected pushforward of $\mu\otimes\mu$ from the infinitely distant past along all histories consistent with~$\omega$:
\begin{cor}\label{cor:nu_expression}
For $\bar\bP$-a.e.\@ $\omega\in\Omega$,
\beqn
\nu^{(2)}_{\omega}(\slot) = \int_{\Omega^{-}} (\mu_{\omega^-}\otimes\mu_{\omega^-})(\slot)\,\rd q_\omega(\omega^-).
\eeqn
\end{cor}
\begin{proof}
Equating first the expressions of $\int h\,\rd\bfP^{(2)}$ in~\eqref{eq:fiber1} and~\eqref{eq:fiber2}, and then applying~\eqref{eq:limit_from_past} to the latter, we obtain
\beqn
\int_\Omega u(\omega)\int_{X\times X} g\,\rd\nu^{(2)}_{\omega}\,\rd\bar\bP(\omega) 
= \int_{\Omega} u(\omega) \int_{\Omega^{-}} \int_{X\times X} g\,\rd(\mu_{\omega^-}\otimes\mu_{\omega^-})\,\rd q_\omega(\omega^-)\,\rd\bar\bP(\omega).
\eeqn
Since the conditional measures $\nu^{(2)}_\omega$ are unique, the claim follows.
\end{proof}

Let us lastly point out that the invariance of $\bfP^{(2)}$ is equivalent to
\beqn
\bar\bE[\varphi^{(2)}(m,\slot)_*\nu^{(2)}_{(\slot)}\,|\,\tau^{-m}\cF\,](\omega) = \nu^{(2)}_{\tau^m\omega}
\eeqn
holding for almost all $\omega$ with respect to $\bar\bP$, for all $m\ge1$. The equation means that
\beqn
\int_\Omega \varphi^{(2)}(m,\omega)_*\nu^{(2)}_{\omega}(g)\, u(\tau^m\omega) \,\rd\bar\bP(\omega) = \int_\Omega \nu^{(2)}_{\tau^m\omega}(g)\, u(\tau^m\omega) \,\rd\bar\bP(\omega) 
\eeqn
holds for all bounded measurable functions~$g:X\times X\to\bR$ and $u:\Omega\to\bR$. It is a good exercise for the interested reader to reprove the invariance of $\bfP^{(2)}$ by verifying the equation above directly, using Corollary~\ref{cor:nu_expression}, Lemma~\ref{lem:limit_from_past} and Lemma~\ref{lem:natural_extension}.


\newpage
\bigskip
\bigskip
\bibliography{Quenched_NA}{}
\bibliographystyle{plainurl}
\newpage

\vspace*{\fill}

\end{document}